\documentclass[11pt,reqno]{amsart}
\usepackage[T1]{fontenc}
\usepackage[utf8]{inputenc}
\usepackage{lmodern}
\usepackage{amsmath,amssymb,amsthm,mathtools,mathrsfs}
\usepackage{enumitem}
\usepackage{graphicx}
\usepackage{microtype}
\usepackage[colorlinks=true,linkcolor=blue,citecolor=blue,urlcolor=blue]{hyperref}

\allowdisplaybreaks
\numberwithin{equation}{section}

\newtheorem{theorem}{Theorem}[section]
\newtheorem{proposition}[theorem]{Proposition}
\newtheorem{lemma}[theorem]{Lemma}
\newtheorem{corollary}[theorem]{Corollary}
\theoremstyle{definition}

\newtheorem{remark}[theorem]{Remark}

\newcommand{\N}{\mathbb N}
\newcommand{\R}{\mathbb R}
\newcommand{\Pp}{\mathcal P}
\newcommand{\E}{\mathbb E}
\newcommand{\Prob}{\mathbb P}
\newcommand{\1}{\mathbf 1}
\newcommand{\Var}{\operatorname{Var}}
\newcommand{\Cov}{\operatorname{Cov}}
\newcommand{\lcm}{\operatorname{lcm}}
\newcommand{\floor}[1]{\left\lfloor #1\right\rfloor}
\newcommand{\norm}[1]{\left\lVert #1\right\rVert}
\newcommand{\abs}[1]{\left\lvert #1\right\rvert}
\newcommand{\ip}[2]{\left\langle #1,#2\right\rangle}
\newcommand{\dd}{\,\mathrm d}
\newcommand{\cH}{\mathcal H}
\newcommand{\cM}{\mathcal M}

\newcommand{\cA}{\mathcal A}
\newcommand{\doi}[1]{\href{https://doi.org/#1}{doi:\nolinkurl{#1}}}
\makeatletter
\@namedef{subjclassname@2020}{\textup{2020} Mathematics Subject Classification}
\makeatother

\title[Functional Limit Theorems for Random Least Common Multiples]
{Functional Limit Theorems for Random Least Common Multiples}

\author{Shaochen Wang}
\address{School of Mathematics, South China University of Technology, Guangzhou 430072, China}
\email{mascwang@scut.edu.cn}

\author{Guangyu Yang}
\address{School of Mathematics and Statistics, Zhengzhou University, Zhengzhou 450001, China}
\email{guangyu@zzu.edu.cn}

\author{Wang Zhou}
\address{Department of Statistics and Data Science, National University of Singapore, Singapore 119077, Singapore}
\email{stazw@nus.edu.sg}



\begin{document}

\begin{abstract}
Let $A_n$ be a subset of $\{1,2,\ldots,n\}$ obtained by retaining each integer independently with fixed probability $\theta\in(0,1)$, and let $L_n$ be the least common multiple of the integers in $A_n$.  We prove a functional large deviation principle, a functional moderate deviation principle, and a Strassen-type functional law of the iterated logarithm for the process $(\log L_{\floor{nt}})_{0\le t\le1}$.  The large deviation rate function is given by an entropy contraction for geometric marks, while the moderate deviation rate function and LIL cluster set are described by the reproducing kernel Hilbert space associated with the Gaussian limit process.
\end{abstract}

\maketitle
\tableofcontents

\section{Introduction}

Let $\theta\in(0,1)$ be fixed, and let $(\xi_j)_{j\ge1}$ be independent Bernoulli random variables with
\[
 \Prob(\xi_j=1)=\theta,
 \qquad
 \Prob(\xi_j=0)=1-\theta.
\]
For $n\in\N$, set
\[
 A_n:=\{j\in\{1,2,\ldots,n\}: \xi_j=1\},
 \qquad
 L_n:=\lcm(A_n),
\]
where $\lcm(\varnothing)=1$.  Thus $L_n$ is the least common multiple of the integers selected from $\{1,\ldots,n\}$.  We also set $A_0=\varnothing$.  This Bernoulli subset model was introduced by Cilleruelo, Ru\'e, \v{S}arka and Zumalac\'arregui \cite{CRSZ}, who proved a weak law of large numbers for $\log L_n$.  Alsmeyer, Kabluchko and Marynych \cite{AKM} later established the corresponding strong law of large numbers and a functional central limit theorem (CLT).  Related models for random least common multiples include independent uniform samples \cite{BMR}, uniformly sampled tuples whose size grows with the ambient range \cite{BIM}, and random $q$-integers \cite{SannaQ}.

The purpose of this paper is to complete the fixed $\theta$ picture beyond laws of large numbers and central limit behavior.  We prove a functional large deviation principle (LDP), a functional moderate deviation principle (MDP), and a Strassen-type functional law of the iterated logarithm (LIL) for the process
\[
t\mapsto\log L_{\lfloor nt\rfloor},\quad 0\le t\le1.
\]
We also derive endpoint consequences and linear functional consequences from these path results.  Large and moderate deviations have been studied extensively in several number-theoretic settings.  For additive arithmetic functions one may consult Mehrdad and Zhu \cite{MZ}.  For Engel, Sylvester, Cantor, and related expansions, see Zhu \cite{ZhuEngel}, Hu \cite{HuSPL}, Fang \cite{FangAltEngel,FangModified}, Fang, Wu and Shang \cite{FangWuShang}, and Wang and Yang \cite{WY}.  The LIL result below is of Strassen-type, in the classical compact cluster set sense of Strassen \cite{Strassen} and its functional developments such as de Acosta \cite{DeAcosta} and Kuelbs \cite{Kuelbs}.

The paper is organized as follows.  Section~\ref{sec:main} states the main theorems and gives the proof strategy.  Section~\ref{sec:numerics} presents numerical checks of the endpoint formulas.  Section~\ref{sec:rates} studies the main properties of rate functions.  Sections~\ref{sec:reduction} and~\ref{sec:prime} provide some arithmetic reduction and prime sum estimates, which will be needed later.  Sections~\ref{sec:LDP},~\ref{sec:MDP}, and~\ref{sec:LIL} prove the three functional results.  Appendix~\ref{sec:appendix-local} contains the local maximal estimate used for tightness, and Appendix~\ref{sec:appendix-euler} gives a supplementary Euler equation for bounded dual fields.

\section{Main results}\label{sec:main}


We adopt the standard terminology for large deviation principles from Dembo and Zeitouni \cite{DZ}, and for weak convergence on Skorokhod spaces from Billingsley \cite{Billingsley}. The main results are presented in three subsections, followed by a discussion of the corresponding proof strategies in the final subsection.

\subsection{Large deviation principle}

Denote
\begin{equation}\label{eq:q-pi}
 q:=1-\theta,
 \qquad
 \pi_k:=\theta q^{k-1},\qquad k\in\N,
\end{equation}
and let $G$ have the geometric distribution $\Prob(G=k)=\pi_k, k\in\N$.  We have
\begin{equation}\label{eq:ctheta}
 c_\theta:=\E G^{-1}
 =\sum_{k\ge1}\frac{\pi_k}{k}
 =\frac{\theta\log(1/\theta)}{1-\theta}.
\end{equation}
This is the limit constant appearing in the weak law of large numbers of \cite{CRSZ} and in the strong law of \cite{AKM}.

For $x\in(0,1]$, let $m(x):=\floor{1/x}$. Define $\cA$ to be the collection of measurable families
\[
 a(x)=(a_0(x),a_1(x),a_2(x),\ldots),\qquad x\in(0,1],
\]
such that, for almost every $x$,
\begin{equation}\label{eq:reduced-kernel-constraints}
 a_k(x)\ge0,
 \qquad
 a_k(x)=0\quad\text{for }k>m(x),
 \qquad
 \sum_{k=0}^{m(x)}a_k(x)=1.
\end{equation}
The coordinate $a_0(x)$ represents the total mass assigned to marks $k>m(x)$, which are invisible on the time interval $[0,1]$.  Set
\begin{align}
 \mathcal R_\theta(a)
 :=\int_0^1
 \bigg[
 &\sum_{k=1}^{m(x)}a_k(x)
 \log\frac{a_k(x)}{\pi_k}
 +a_0(x)
 \log\frac{a_0(x)}{q^{m(x)}}
 \bigg]\dd x,
 \label{eq:reduced-entropy}
\end{align}
with the convention $0\log0=0$.  Thus $\mathcal R_\theta(a)$ is the average, over rescaled prime locations $x$, of the relative entropy of the visible probability vector together with the invisible tail mass $a_0(x)$, relative to the geometric law $(\pi_k)_{k\ge1}$.  For relative entropy, see \cite[Section~6.2]{DZ}.  Changing the distribution of the geometric mark at location $x$ costs relative entropy, and the map $\mathcal T$ turns those modified marks into a path.  This path is
\begin{equation}\label{eq:reduced-map}
 (\mathcal T a)(t)
 :=\sum_{k\ge1}\int_0^{t/k}a_k(x)\dd x,
 \qquad 0\le t\le1.
\end{equation}
Define
\begin{equation}\label{eq:functional-rate-main}
 I_\theta(f)
 :=\inf\left\{
 \mathcal R_\theta(a):a\in\cA,\ \mathcal T a=f
 \right\},
 \qquad f\in C[0,1],
\end{equation}
where the infimum of the empty set is $+\infty$.

For $0\le t\le1$, set
\begin{equation}\label{eq:Xn}
 X_n(t):=\frac{\log L_{\floor{nt}}}{n},
\end{equation}
and let $\widehat X_n$ denote the polygonal interpolation of the values $X_n(j/n)$, $0\le j\le n$. The functional LDP can now be stated in terms of this entropy contraction.

\begin{theorem}\label{thm:FLDP}
The sequence $(\widehat X_n)_{n\ge1}$ satisfies an LDP in $C[0,1]$, equipped with the uniform norm, with speed
\[
 v_n:=\frac{n}{\log n}
\]
and good rate function $I_\theta$ defined by \eqref{eq:functional-rate-main}.  The step processes $(X_n)_{n\ge1}$ satisfy the same LDP in $D[0,1]$ equipped with the Skorokhod $J_1$-topology, with rate $I_\theta$ on $C[0,1]$ and $+\infty$ on $D[0,1]\setminus C[0,1]$.
\end{theorem}

For $\lambda\in\R$, define
\begin{equation}\label{eq:endpoint-cumulant}
 \Lambda_\theta(\lambda)
 :=\sum_{j=1}^{\infty}\frac1{j(j+1)}
 \log\left(q^j+(1-q^j)e^\lambda\right).
\end{equation}

The endpoint contraction has the following one-dimensional Legendre form.

\begin{corollary}\label{cor:endpoint-LDP}
The sequence $(n^{-1}\log L_n)_{n\ge1}$ satisfies an LDP on $\R$ with speed $n/\log n$ and rate
\begin{equation}\label{eq:endpoint-rate-Legendre}
 I_\theta^{\mathrm{end}}(y)
 =\sup_{\lambda\in\R}
 \{\lambda y-\Lambda_\theta(\lambda)\}.
\end{equation}
For every $y\in(0,1)$ there is a unique $\lambda_y\in\R$ satisfying
\begin{equation}\label{eq:lambda-y-equation}
 y=\sum_{j=1}^{\infty}\frac1{j(j+1)}
 \frac{(1-q^j)e^{\lambda_y}}
 {q^j+(1-q^j)e^{\lambda_y}},
\end{equation}
and
\begin{equation}\label{eq:endpoint-rate-explicit}
 I_\theta^{\mathrm{end}}(y)
 =y\lambda_y-\Lambda_\theta(\lambda_y),
 \qquad 0<y<1.
\end{equation}
More explicitly, define
\begin{equation}\label{eq:endpoint-optimal-profile}
 u_y(x):=
 \frac{(1-q^{m(x)})e^{\lambda_y}}
 {q^{m(x)}+(1-q^{m(x)})e^{\lambda_y}},
 \qquad 0<x\le1.
\end{equation}
Then $\int_0^1u_y(x)\dd x=y$ and
\begin{align}
 I_\theta^{\mathrm{end}}(y)
 =\int_0^1\bigg[
 u_y(x)\log\frac{u_y(x)}{1-q^{m(x)}}
 +(1-u_y(x))
 \log\frac{1-u_y(x)}{q^{m(x)}}
 \bigg]\dd x.
 \label{eq:endpoint-binary-entropy}
\end{align}
The boundary values are
\begin{equation}\label{eq:endpoint-boundaries}
 I_\theta^{\mathrm{end}}(0)=+\infty,
 \qquad
 I_\theta^{\mathrm{end}}(1)
 =-\sum_{j=1}^{\infty}\frac1{j(j+1)}\log(1-q^j),
\end{equation}
and $I_\theta^{\mathrm{end}}(y)=+\infty$ for $y\notin[0,1]$.  Its unique zero is $c_\theta$.
\end{corollary}

\subsection{Moderate deviation principle}

We next state the moderate deviation principle.  Let $(b_n)_{n\geq 1}$ satisfy
\begin{equation}\label{eq:MDP-scale}
 b_n\longrightarrow\infty,
 \qquad
 b_n\sqrt{\frac{\log n}{n}}\longrightarrow0.
\end{equation}
Define
\begin{equation}\label{eq:Yn}
 Y_n(t):=
 \frac{\log L_{\floor{nt}}-\E\log L_{\floor{nt}}}
 {b_n\sqrt{n\log n}},
 \qquad 0\le t\le1,
\end{equation}
and let $\widehat Y_n$ be its polygonal interpolation.

The covariance kernel of the Gaussian process in the functional central limit theorem of \cite{AKM} is
\begin{align}
 C_\theta(s,t)
 &:=\int_0^1
 \Cov\bigl(\1_{\{xG\le s\}},\1_{\{xG\le t\}}\bigr)\dd x
 \label{eq:cov-integral}\\
 &=\sum_{k\ge1}\pi_k
 \left(\frac{s}{k}\wedge\frac{t}{k}\right)
 -\sum_{k,l\ge1}\pi_k\pi_l
 \left(\frac{s}{k}\wedge\frac{t}{l}\right).
 \label{eq:cov-series}
\end{align}
Let $\cH_\theta$ denote the reproducing kernel Hilbert space associated
with the covariance kernel $C_\theta$, in the sense of Aronszajn
\cite{Aronszajn}.  It is known that for $g\in C[0,1]$,
\[
\frac12\norm{g}_{\cH_\theta}^2
=
\sup_{m\ge1,\ t_1,\ldots,t_m\in[0,1],\ a_1,\ldots,a_m\in\mathbb R}
\bigg\{
\sum_{i=1}^m a_i g(t_i)
-\frac12\sum_{i,j=1}^m a_i a_j C_\theta(t_i,t_j)
\bigg\};
\]
by convention, the supremum is $+\infty$ if $g\notin\cH_\theta$.
With this notation, the moderate deviation principle has the following Gaussian quadratic form.

\begin{theorem}\label{thm:FMDP}
Under \eqref{eq:MDP-scale}, the sequence $(\widehat Y_n)_{n\ge1}$ satisfies an MDP in $C[0,1]$, equipped with the uniform norm, with speed $b_n^2$ and good rate function
\begin{equation}\label{eq:MDP-rate}
 J_\theta(f)
 :=\begin{cases}
 \dfrac12\norm{f}_{\cH_\theta}^2,
 &f\in\cH_\theta,\\[1ex]
 +\infty,&f\notin\cH_\theta.
 \end{cases}
\end{equation}
The step processes $(Y_n)_{n\ge1}$ satisfy the same MDP in $D[0,1]$ with the $J_1$-topology.
\end{theorem}

Set
\begin{equation}\label{eq:sigma-theta}
 \sigma_\theta^2:=C_\theta(1,1)
 =\sum_{j=1}^{\infty}
 \frac{q^j(1-q^j)}{j(j+1)}.
\end{equation}
Indeed, for $x\in(1/(j+1),1/j]$ one has
$\Prob(xG\le1)=1-q^j$, and integrating the Bernoulli variance over these intervals gives the series in \eqref{eq:sigma-theta}. Evaluating the functional MDP at time $1$ gives the endpoint form.

\begin{corollary}\label{cor:endpoint-MDP}
Under the assumptions in  \eqref{eq:MDP-scale},
\[
 \frac{\log L_n-\E\log L_n}
 {b_n\sqrt{n\log n}}
\]
satisfies an MDP on $\R$ with speed $b_n^2$ and rate
\begin{equation}\label{eq:endpoint-MDP-rate}
 J_\theta^{\mathrm{end}}(x)
 =\frac{x^2}{2\sigma_\theta^2},\quad x\in\R.
\end{equation}
\end{corollary}

\subsection{The law of the iterated logarithm}

For $n\ge3$, define the iterated logarithm normalization
\begin{equation}\label{eq:LIL-normalization}
 a_n:=\sqrt{2n\log n\log\log n}
\end{equation}
and the centered process
\begin{equation}\label{eq:LIL-process}
 \Xi_n(t):=
 \frac{\log L_{\floor{nt}}-\E\log L_{\floor{nt}}}{a_n},
 \qquad 0\le t\le1.
\end{equation}
Let $\widehat\Xi_n$ be the polygonal interpolation of $\Xi_n$, and put
\begin{equation}\label{eq:LIL-cluster-set}
 \mathcal K_\theta:=\left\{f\in\cH_\theta:
 \norm{f}_{\cH_\theta}\le1\right\}.
\end{equation}

The compact law of the iterated logarithm is formulated as a cluster set statement.

\begin{theorem}\label{thm:FLIL}
Write
\[
 \|f\|_\infty:=\sup_{0\le t\le1}|f(t)|,
 \qquad
 \operatorname{dist}_\infty(f,B):=
 \inf_{g\in B}\|f-g\|_\infty.
\]
There exists an event $\Omega_0$ with $\Prob(\Omega_0)=1$ such that, for every $\omega\in\Omega_0$, the sequence
$(\widehat\Xi_n(\omega))_{n\ge3}$ is relatively compact in $C[0,1]$ and its cluster set is exactly $\mathcal K_\theta$.  Equivalently, on $\Omega_0$,
\[
 \operatorname{dist}_{\infty}
 (\widehat\Xi_n,\mathcal K_\theta)\longrightarrow0,
\]
and, for every $f\in\mathcal K_\theta$, there is a subsequence
$n_j=n_j(f,\omega)\to\infty$ such that
\[
 \norm{\widehat\Xi_{n_j}-f}_\infty\longrightarrow0.
\]
The step processes $(\Xi_n)_{n\geq 3}$ have the same cluster set in $D[0,1]$ endowed with the $J_1$-topology.
\end{theorem}

The endpoint and linear functional consequences follow by applying continuous maps to the compact cluster set.

\begin{corollary}\label{cor:endpoint-LIL}
Almost surely, we have
\begin{equation}\label{eq:endpoint-LIL-limsup}
 \limsup_{n\to\infty}
 \frac{\log L_n-\E\log L_n}
 {\sqrt{2n\log n\log\log n}}
 =\sigma_\theta,
\end{equation}
and
\begin{equation}\label{eq:endpoint-LIL-liminf}
 \liminf_{n\to\infty}
 \frac{\log L_n-\E\log L_n}
 {\sqrt{2n\log n\log\log n}}
 =-\sigma_\theta.
\end{equation}
More generally, if $\ell$ is a continuous linear functional on $C[0,1]$, then
\begin{equation}\label{eq:linear-functional-LIL}
 \limsup_{n\to\infty}\ell(\widehat\Xi_n)
 =\norm{\ell}_{\cH_\theta^\ast},
 \qquad
 \liminf_{n\to\infty}\ell(\widehat\Xi_n)
 =-\norm{\ell}_{\cH_\theta^\ast}
\end{equation}
almost surely, where $\norm{\ell}_{\cH_\theta^\ast}$ denotes the operator norm of the restriction of $\ell$ to the RKHS $\cH_\theta$, namely
\[
 \norm{\ell}_{\cH_\theta^\ast}^2
 :=\sup_{\norm{f}_{\cH_\theta}\le1}|\ell(f)|^2.
\]
This number is finite because the RKHS unit ball is compact in $C[0,1]$ and $\ell$ is continuous for the uniform norm.
\end{corollary}

\begin{remark}\label{rem:Poisson-regimes}
The present paper keeps the retention probability $\theta$ fixed in $(0,1)$.  If $\theta=\theta_n$ varies with $n$, the large prime mechanism changes thoroughly.  In the sparse and nearly complete regimes in the setting of \cite{AKM}, the natural limits are Poissonian rather than Gaussian, and the proofs for corresponding functional limit theorems use point process and Poisson array approximations instead of the geometric mark entropy and RKHS structure developed here.  Functional Poisson limits and the corresponding large and moderate deviations for the varying $\theta_n$ model are treated in the companion working paper \cite{WYZ}.
\end{remark}

\subsection{Proof strategies}\label{subsec:proof-strategies}
For convenience, the letter $p$ always denotes a prime number throughout the paper.
The proofs start from the same idea that underlies the central limit theorem in \cite{AKM}, where the  key arithmetic identity is the von Mangoldt expansion. Recall that the von Mangoldt function is defined by
\[
 \Lambda(m)=
 \begin{cases}
 \log p,&m=p^r\text{ for a prime }p\text{ and an integer }r\ge1,\\
 0,&\text{otherwise},
 \end{cases}
\]
see, for example, Apostol \cite[Chapter~4]{Apostol}.  If $I_A(m):=\1_{\{A\cap m\N\ne\varnothing\}}$ for a finite set $A\subset\N$, then
\[
 \log\lcm(A)=\sum_{m\ge1}\Lambda(m)I_A(m),
\]
which is also stated in \cite[Lemma~2.1]{AKM}.  The terms coming from primes at most $\sqrt n$, together with the higher prime powers, are uniformly $O(\sqrt n)$.  They are negligible for all three asymptotic scales considered here.

After this reduction, only primes $p>\sqrt n$ remain.  For two such primes, the sets of multiples $p\N\cap[1,n]$ and $q\N\cap[1,n]$ are disjoint.  Hence the large prime part splits into independent prime coordinates.  For such a prime $p$, let $G_{n,p}$ be the first selected multiple index of $p$, constructed on an enlarged probability space as in \eqref{eq:geometric-mark}.  Then $G_{n,p}$ has the geometric law \eqref{eq:q-pi}, and the large prime process can be written as
\[
 S_n(t)=\sum_{\sqrt n<p\le n}\log p\,
 \1_{\{(p/n)G_{n,p}\le t\}}.
\]
Thus the problem becomes a triangular array of independent weighted threshold processes.  The prime number theorem then converts sums over $p$ into integrals over $x=p/n$.  This conversion gives the mean and covariance in the central limit theorem.

For the functional LDP, the same independent representation gives finite-dimensional logarithmic moment generating functions.  The G\"artner--Ellis theorem \cite[Theorem~2.3.6]{DZ} yields finite-dimensional large deviations.  The limiting rate is first obtained as a projective Legendre transform and then identified with the entropy contraction in Theorem~\ref{thm:FLDP}.  The functional lift uses an exponential equicontinuity estimate.  In the LDP case this estimate is readily obtained because the process is monotone.

For the functional MDP, the centered large prime process is expanded to second order.  The quadratic term is exactly the covariance kernel $C_\theta$, while the cubic remainder is controlled by the assumption \eqref{eq:MDP-scale}.  This gives finite-dimensional moderate deviations with the Gaussian quadratic rate.  Since centering destroys monotonicity, the functional lift uses the local maximal estimate for weighted threshold processes proved in Proposition~\ref{prop:local-threshold} in Appendix~\ref{sec:appendix-local}.  The RKHS representation of the rate follows from a Hilbert space contraction.

The functional LIL uses the MDP at the critical scale $b_n=\sqrt{2\log\log n}$.  The upper cluster set inclusion follows by applying the MDP on a geometric subsequence and then filling the gaps by a block interpolation estimate.  The lower inclusion is obtained from independent lacunary prime blocks.  These blocks have the same MDP rate, and the second Borel--Cantelli lemma provides subsequences approaching every point of the RKHS unit ball.  Finally, the deterministic arithmetic approximation carries all conclusions from the large prime process over to $\log L_{\lfloor nt\rfloor}$.

\section{Numerical simulations}\label{sec:numerics}

This section gives some numerical validations for the endpoint LDP and MDP formulas.  The computations use the independent large prime process
\[
 S_n(t)=\sum_{\sqrt n<p\le n}\log p\,
 \1_{\{(p/n)G_{n,p}\le t\}},
\]
where the variables $G_{n,p}$ are independent geometric marks with law \eqref{eq:q-pi}.  By Lemma~\ref{lem:deterministic-approximation}, $S_n(t)$ and $\log L_{\lfloor nt\rfloor}$ differ by only $O(\sqrt n)$ uniformly in $t$.  This deterministic error is negligible on both the LDP and MDP scales considered here.  The plots here therefore illustrate the  large prime mechanism that appears in the proofs.

Throughout this section, we set $\theta=1/2$.  The numerical values used in the plots are
\[
 c_\theta\approx 0.693147,\qquad \sigma_\theta^2\approx 0.169899.
\]
The finite $n$ cumulants below are computed from exact product formulas for the independent prime coordinates.  The rare probability estimates are obtained by exponential tilting.  In that procedure the simulation is performed under the tilted Bernoulli law in \eqref{eq:finite-n-tilted-probability}, and the likelihood ratio is used to recover probabilities under the original law.  The displayed Monte Carlo standard errors on the rate scale are computed from the weighted second moment by the delta method. The limiting endpoint rate is
\begin{equation}\label{eq:numerical-endpoint-rate}
 I_\theta^{\mathrm{end}}(y)
 =\sup_{\lambda\in\R}\{\lambda y-\Lambda_\theta(\lambda)\},
 \quad
 \Lambda_\theta(\lambda)
 =\sum_{j\ge1}\frac{1}{j(j+1)}
 \log\{q^j+(1-q^j)e^\lambda\}.
\end{equation}
For $y\in(0,1)$, the optimizing parameter is the unique solution of
\begin{equation}\label{eq:numerical-saddle-point}
 y=\Lambda_\theta'(\lambda)
 =\sum_{j\ge1}\frac{1}{j(j+1)}
 \frac{(1-q^j)e^\lambda}{q^j+(1-q^j)e^\lambda}.
\end{equation}
Figure~\ref{fig:endpoint-rate-quadratic} shows the endpoint rate and its local MDP approximation.  The quadratic approximation is accurate near the typical value and separates from the full LDP rate in the tails.

\begin{figure}[htbp]
\centering
\includegraphics[width=.74\textwidth]{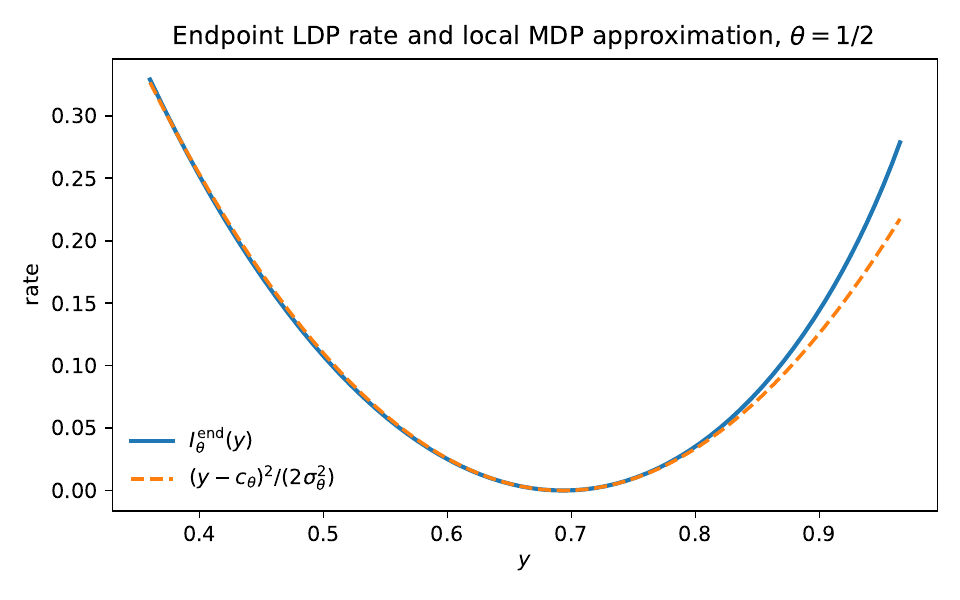}
\caption{Endpoint rate $I_\theta^{\mathrm{end}}$ and its quadratic MDP approximation for $\theta=1/2$.}
\label{fig:endpoint-rate-quadratic}
\end{figure}

The finite $n$ large prime cumulant is explicit.  If we denote $k_p=\lfloor n/p\rfloor$, and define
\begin{equation}\label{eq:finite-n-layer-cgf}
 \Lambda_{\theta,n}^{\mathrm{LP}}(\lambda)
 :=\frac{\log n}{n}\sum_{\sqrt n<p\le n}
 \log\left[
 q^{k_p}+(1-q^{k_p})
 \exp\left\{\lambda\frac{\log p}{\log n}\right\}
 \right]
\end{equation}
and
\begin{equation}\label{eq:finite-n-layer-rate}
 I_{\theta,n}^{\mathrm{LP}}(y)
 :=\sup_{\lambda\in\R}
 \{\lambda y-\Lambda_{\theta,n}^{\mathrm{LP}}(\lambda)\}.
\end{equation}
The proof of the endpoint LDP gives $\Lambda_{\theta,n}^{\mathrm{LP}}(\lambda)\to\Lambda_\theta(\lambda)$ and hence convergence of the associated Legendre transforms on compact subintervals of $(0,1)$.  Figure~\ref{fig:finite-n-rate} displays this convergence.  

\begin{figure}[htbp]
\centering
\includegraphics[width=.74\textwidth]{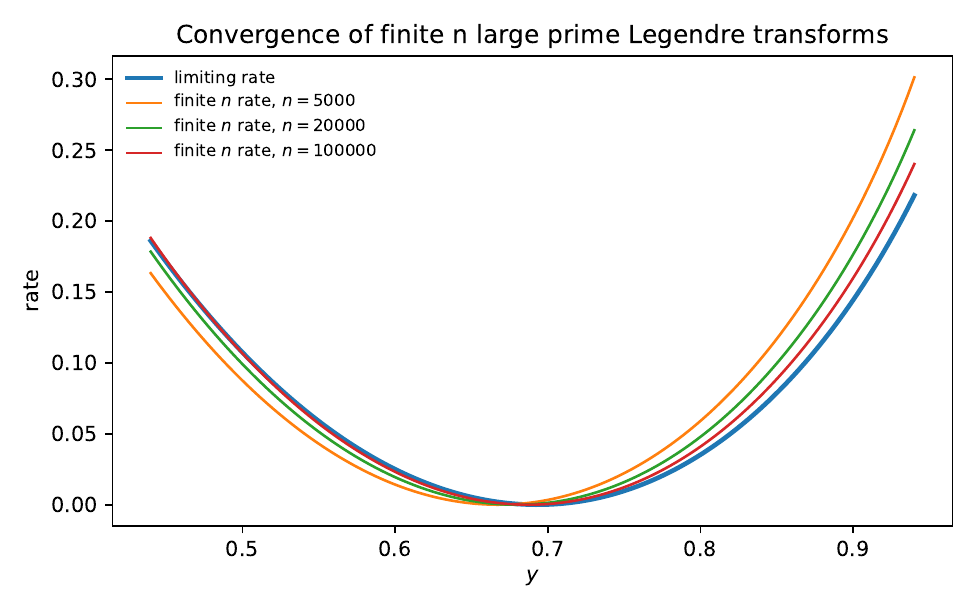}
\caption{Exact finite $n$ Legendre transforms for the large prime process and their limiting endpoint rate.}
\label{fig:finite-n-rate}
\end{figure}

We next check the far tails by exponential tilting.  For a bin $B=[y-h,y+h]$ lying on one side of the finite $n$ mean, let $z_*$ be the endpoint of $B$ closest to that mean and choose $\lambda_*$ so that $(\Lambda_{\theta,n}^{\mathrm{LP}})'(\lambda_*)=z_*$.  Under the tilted law, the success probability attached to a prime $p$ is
\begin{equation}\label{eq:finite-n-tilted-probability}
 \alpha_{n,p}(\lambda_*)
 =\frac{(1-q^{k_p})
 \exp\{\lambda_*\log p/\log n\}}
 {q^{k_p}+(1-q^{k_p})
 \exp\{\lambda_*\log p/\log n\}}.
\end{equation}
The likelihood ratio is evaluated in logarithmic form.  We also report the first order saddlepoint approximation
\begin{equation}\label{eq:finite-n-saddlepoint-bin}
\begin{aligned}
 \Prob\{n^{-1}S_n(1)\in B\}
 &\approx
 \frac{1-e^{-v_n|\lambda_*|\,2h}}
 {|\lambda_*|\sqrt{2\pi v_n
 (\Lambda_{\theta,n}^{\mathrm{LP}})''(\lambda_*)}}\\
 &\qquad{}\times
 \exp\{-v_n I_{\theta,n}^{\mathrm{LP}}(z_*)\},
 \qquad v_n=\frac n{\log n}.
\end{aligned}
\end{equation}
This approximation serves solely as a numerical benchmark. Figure~\ref{fig:importance-sampling} and Table~\ref{tab:importance-sampling} show that the importance sampling (IS) estimates agree with the finite $n$ Legendre exponent and with the saddlepoint values within the plotted Monte Carlo error.  This confirms that the rare events are governed by the same large prime mechanism as in the proof.

\begin{figure}[htbp]
\centering
\includegraphics[width=.74\textwidth]{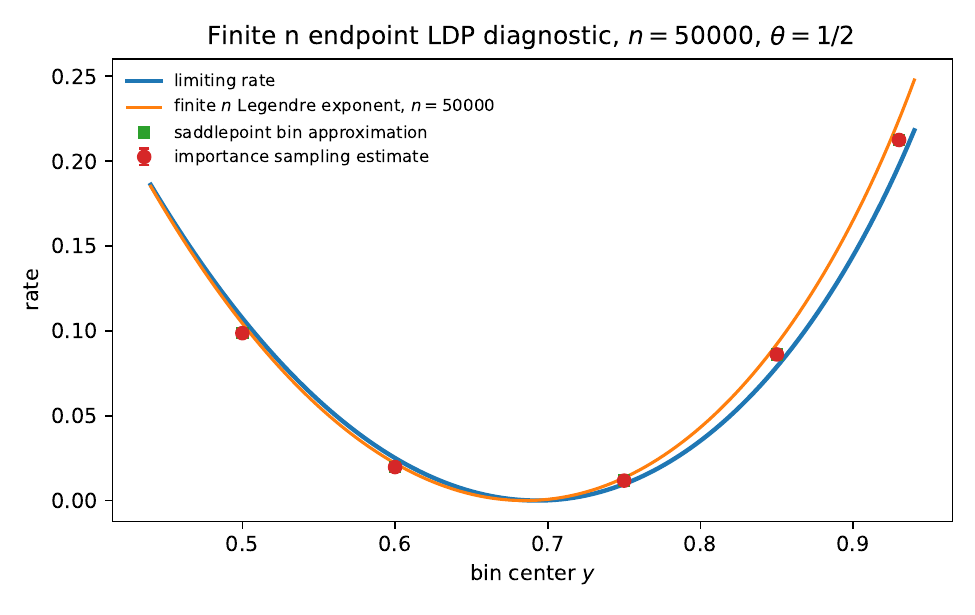}
\caption{Finite $n$ endpoint LDP diagnostic for the large prime process.  Squares show the first order saddlepoint bin approximation, and circles show IS estimates.  Error bars are approximate $95\%$ Monte Carlo intervals on the rate scale.}
\label{fig:importance-sampling}
\end{figure}

\begin{table}[htbp!]
\centering
\caption{IS and saddlepoint checks for the large prime layer at $n=50000$, $\theta=1/2$, bin half-width $h=0.006$, and 6000 tilted replications for each bin.  The finite $n$ exponent is the infimum of the exact finite $n$ Legendre transform over the bin.}
\label{tab:importance-sampling}
\small
\begin{tabular}{@{}c c c c c c@{}}
\hline
$y$ & Legendre exponent & saddlepoint & IS estimate & rate s.e. & tilted hit rate \\
\hline
0.50 & 0.0977 & 0.0986 & 0.0986 & 0.00002 & 0.489 \\
0.60 & 0.0190 & 0.0197 & 0.0197 & 0.00001 & 0.472 \\
0.75 & 0.0110 & 0.0117 & 0.0117 & 0.00001 & 0.490 \\
0.85 & 0.0852 & 0.0861 & 0.0861 & 0.00002 & 0.493 \\
0.93 & 0.2115 & 0.2125 & 0.2125 & 0.00002 & 0.500 \\
\hline
\end{tabular}
\end{table}

The MDP can be checked without estimating rare probabilities.  For $u_{n,p}(\lambda)=\lambda b_n\log p/\sqrt{n\log n}$, define the centered finite $n$ logarithmic moment generating function
\begin{equation}\label{eq:finite-n-MDP-cgf}
\begin{aligned}
 K_{\theta,n,b_n}(\lambda)
 :=\frac1{b_n^2}\sum_{\sqrt n<p\le n}
 \bigg[&\log\{q^{k_p}+(1-q^{k_p})e^{u_{n,p}(\lambda)}\}-(1-q^{k_p})u_{n,p}(\lambda)\bigg].
\end{aligned}
\end{equation}
Theorem~\ref{thm:FMDP} predicts
$K_{\theta,n,b_n}(\lambda)\to\lambda^2\sigma_\theta^2/2$ under \eqref{eq:MDP-scale}.  Figure~\ref{fig:mdp-cumulants} shows this convergence for $b_n=(\log n)^{1/4}$.

\begin{figure}[htbp]
\centering
\includegraphics[width=.74\textwidth]{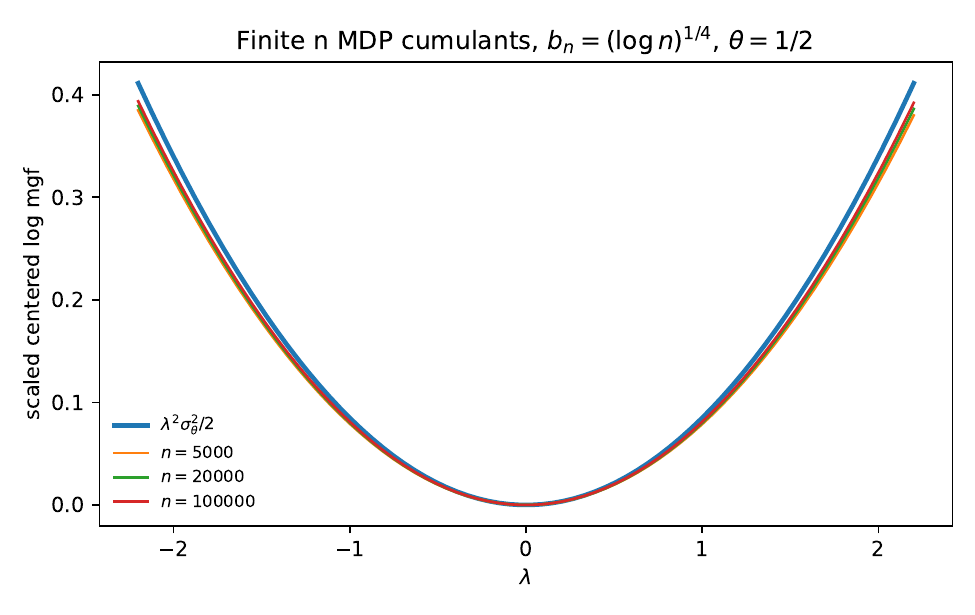}
\caption{Centered finite $n$ logarithmic moment generating functions and the limiting MDP quadratic cumulant for $\theta=1/2$.}
\label{fig:mdp-cumulants}
\end{figure}

\section{Structure of the rate functions}\label{sec:rates}

This section gives alternative formulas and the main properties of the two functional rate functions.  These results are also used in the proofs of Theorems~\ref{thm:FLDP} and~\ref{thm:FMDP}.  A supplementary Euler-type optimality equation for possible dual maximizers of the large deviation rate is placed in Appendix~\ref{sec:appendix-euler}.  It is not needed in the main arguments.

\subsection{Entropy and convex-dual representations of the LDP rate}

Let $\mathscr K$ be the collection of measurable probability kernels
\[
 x\longmapsto\nu_x=(\nu_x(k))_{k\ge1}
\]
from $[0,1]$ to $\N$.  Define
\begin{equation}\label{eq:full-entropy}
 \mathcal E_\theta(\nu)
 :=\int_0^1H(\nu_x\mid\pi)\dd x,
 \qquad
 H(\nu_x\mid\pi)
 :=\sum_{k\ge1}\nu_x(k)
 \log\frac{\nu_x(k)}{\pi_k},
\end{equation}
and
\begin{equation}\label{eq:full-map}
 (T\nu)(t)
 :=\int_0^1\sum_{k\ge1}
 \1_{\{xk\le t\}}\nu_x(k)\dd x.
\end{equation}

We first compare the reduced entropy formula with the full geometric kernel entropy.

\begin{proposition}\label{prop:entropy-equivalence}
For every $f\in C[0,1]$,
\begin{equation}\label{eq:full-entropy-rate}
 I_\theta(f)
 =\inf\left\{
 \mathcal E_\theta(\nu):\nu\in\mathscr K,\ T\nu=f
 \right\}.
\end{equation}
The full kernel contraction in \eqref{eq:full-entropy-rate} is equal to the reduced formula \eqref{eq:functional-rate-main}.  Whenever the value in \eqref{eq:full-entropy-rate} is finite, the infimum is attained.
\end{proposition}

\begin{proof}
Fix $x\in(0,1]$ and put $m=m(x)$.  Given a probability vector $\nu_x$, define
\[
 a_k(x):=\nu_x(k),\quad 1\le k\le m,
 \qquad
 a_0(x):=\sum_{k>m}\nu_x(k).
\]
Since $\pi_k=\theta q^{k-1}$,
\begin{equation}\label{eq:geometric-tail-mass}
 \sum_{k>m}\pi_k=q^m.
\end{equation}
If $a_0(x)>0$, define probability measures on the tail $\{m+1,m+2,\ldots\}$ by
\[
 \widetilde\nu_x(k):=\frac{\nu_x(k)}{a_0(x)},
 \qquad
 \widetilde\pi_m(k):=\frac{\pi_k}{q^m},
 \qquad k>m.
\]
If $a_0(x)=0$, the conditional law $\widetilde\nu_x$ is immaterial.  For definiteness one may take $\widetilde\nu_x=\widetilde\pi_m$, and all displayed terms multiplied by $a_0(x)$ are interpreted as zero.  Splitting the entropy into the visible coordinates, the total tail mass and the conditional distribution inside the tail gives the exact chain rule
\begin{align}\label{eq:entropy-chain-tail}
 H(\nu_x\mid\pi)
 =&\sum_{k=1}^{m}a_k(x)
 \log\frac{a_k(x)}{\pi_k}
 +a_0(x)\log\frac{a_0(x)}{q^m}
 +a_0(x)H(\widetilde\nu_x\mid\widetilde\pi_m).
\end{align}
The last term in \eqref{eq:entropy-chain-tail} is nonnegative.
Moreover, for $0\le t\le1$, the inequality $xk\le t$ implies $k\le m(x)$.  Thus the path $T\nu$ depends only on the visible coordinates and satisfies
\[
 T\nu=\mathcal T a.
\]
After integration in $x$, \eqref{eq:entropy-chain-tail} yields
\[
 \mathcal E_\theta(\nu)\ge\mathcal R_\theta(a).
\]

Conversely, given $a\in\cA$, define the lifted kernel
\begin{equation}\label{eq:optimal-tail-lift}
 \nu_x(k):=
 \begin{cases}
 a_k(x),&1\le k\le m(x),\\[1mm]
 a_0(x)\pi_k/q^{m(x)},&k>m(x).
 \end{cases}
\end{equation}
The functions in \eqref{eq:optimal-tail-lift} are measurable because $m(x)$ is measurable and piecewise constant.  The kernel is a probability kernel by \eqref{eq:reduced-kernel-constraints} and \eqref{eq:geometric-tail-mass}.  Its conditional tail law is $\widetilde\pi_{m(x)}$, so the last term in \eqref{eq:entropy-chain-tail} vanishes.  Hence
\[
 T\nu=\mathcal T a,
 \qquad
 \mathcal E_\theta(\nu)=\mathcal R_\theta(a),
\]
which proves the equality of the two contractions.

It remains to verify attainment.  Associate with $\nu$ the probability measure
\[
 \gamma_\nu(\dd x,\{k\})=\nu_x(k)\dd x
\]
on the Polish space $E=[0,1]\times\N$.  Its first marginal is Lebesgue measure and
\[
 \mathcal E_\theta(\nu)
 =H(\gamma_\nu\mid \operatorname{Leb}\otimes\pi).
\]
Relative entropy is lower semicontinuous for weak convergence, and its sublevel sets relative to a fixed reference probability measure are tight and weakly compact.  See, for example, \cite[Section~6.2]{DZ}.  Note that, if $\gamma_n\Rightarrow\gamma$ and the first marginal of every $\gamma_n$ is Lebesgue measure, then testing against bounded continuous functions depending only on the first coordinate shows that the first marginal of $\gamma$ is also Lebesgue measure.  In other words, the set of probability measures whose first marginal is Lebesgue measure is weakly closed.  Furthermore, since $[0,1]\times\N$ is a standard Borel space, every measure in this closed set admits a measurable disintegration with respect to its first marginal.  Consequently,
\begin{equation}\label{eq:kernel-compactness}
 \{\nu\in\mathscr K:\mathcal E_\theta(\nu)\le M\}
\end{equation}
is compact when kernels are identified with their joint measures $\gamma_\nu$.

For fixed $t\in[0,1]$, set
\[
 B_t:=\{(x,k)\in E:xk\le t\}.
\]
Its boundary is contained in the countable set $\{(t/k,k):k\ge1\}$.  Every measure with first marginal Lebesgue measure gives this boundary zero mass.  The portmanteau theorem consequently implies that $\nu\mapsto(T\nu)(t)=\gamma_\nu(B_t)$ is continuous on the kernel space.  Hence the constraint $T\nu=f$ is closed when imposed at all times.  Intersecting this closed constraint with a sufficiently large compact entropy sublevel and using a minimizing sequence proves attainment whenever the infimum is finite.
\end{proof}

Let $\cM_{\mathrm{at}}[0,1]$ denote the finite signed atomic measures on $[0,1]$.  For $\mu\in\cM_{\mathrm{at}}[0,1]$ and $x\in(0,1]$, define
\begin{equation}\label{eq:Zmu}
 Z_\mu(x)
 :=q^{m(x)}+
 \sum_{k=1}^{m(x)}\pi_k
 \exp\{\mu([xk,1])\},
\end{equation}
and
\begin{equation}\label{eq:functional-cumulant-measure}
 \mathfrak L_\theta(\mu)
 :=\int_0^1\log Z_\mu(x)\dd x.
\end{equation}
The term $q^{m(x)}$ in \eqref{eq:Zmu} is the total mass of the invisible marks $k>m(x)$.

The same rate also has the following finite dimensional convex-dual representation.

\begin{proposition}\label{prop:dual-rate}
For every $f\in C[0,1]$,
\begin{equation}\label{eq:functional-dual-rate}
 I_\theta(f)
 =\sup_{\mu\in\cM_{\mathrm{at}}[0,1]}
 \bigg\{
 \int_{[0,1]}f(t)\mu(\dd t)-\mathfrak L_\theta(\mu)
 \bigg\}.
\end{equation}
Equivalently,
\begin{align}
 I_\theta(f)
 =\sup_{m,\boldsymbol t,\lambda}
 \bigg\{
 &\sum_{j=1}^m\lambda_jf(t_j)
 \nonumber\\
 &-\int_0^1\log\bigg[
 q^{m(x)}+
 \sum_{k=1}^{m(x)}\pi_k
 \exp\bigg\{\sum_{j=1}^m
 \lambda_j\1_{\{xk\le t_j\}}\bigg\}
 \bigg]\dd x
 \bigg\},
 \label{eq:projective-rate}
\end{align}
where the supremum is over $m\ge1$, $0\le t_1<\cdots<t_m\le1$, and $\lambda\in\R^m$.

For a fixed atomic measure $\mu$, the pointwise optimizer in the finite-state Gibbs variational formula used in \eqref{eq:finite-gibbs} is the reduced kernel
\begin{equation}\label{eq:tilted-reduced-kernel}
 a_k^\mu(x)
 =\frac{\pi_k e^{\mu([xk,1])}}{Z_\mu(x)},
 \quad 1\le k\le m(x),
 \qquad
 a_0^\mu(x)=\frac{q^{m(x)}}{Z_\mu(x)}.
\end{equation}
\end{proposition}

\begin{proof}
Fix a finite set $F=\{t_1,\ldots,t_m\}$.  Define
\begin{equation}\label{eq:JF}
 \mathcal J_F(z)
 :=\inf\left\{
 \mathcal E_\theta(\nu):
 ((T\nu)(t_1),\ldots,(T\nu)(t_m))=z
 \right\}.
\end{equation}
The compactness and continuity established in the proof of Proposition~\ref{prop:entropy-equivalence} show that $\mathcal J_F$ is lower semicontinuous, has compact level sets, and attains every finite value.  It is convex because relative entropy is convex and the constraint map is linear.

For $\lambda\in\R^m$, we use the elementary Gibbs variational formula; see, for example, Dupuis and Ellis \cite[Proposition~1.4.2]{DupuisEllis}:
\begin{equation}\label{eq:finite-gibbs}
 \log\sum_i p_i e^{u_i}=\sup_{a_i\ge0,\,\sum_i a_i=1}\bigg\{\sum_i a_i u_i-\sum_i a_i\log\frac{a_i}{p_i}\bigg\},
\end{equation}
with the usual convention that $0\log0=0$.  Applying \eqref{eq:finite-gibbs} at each fixed $x$ gives
\begin{align}
 &\log\bigg[
 q^{m(x)}+
 \sum_{k=1}^{m(x)}\pi_k
 \exp\bigg\{\sum_{j=1}^m\lambda_j
 \1_{\{xk\le t_j\}}\bigg\}
 \bigg]
 \nonumber\\
 &\quad=\sup_{a(x)}\bigg\{
 \sum_{j=1}^m\lambda_j
 \sum_{k=1}^{m(x)}a_k(x)\1_{\{xk\le t_j\}}
 \nonumber\\
 &\hspace{37mm}-
 \sum_{k=1}^{m(x)}a_k(x)\log\frac{a_k(x)}{\pi_k}
 -a_0(x)\log\frac{a_0(x)}{q^{m(x)}}
 \bigg\}.
 \label{eq:pointwise-Gibbs}
\end{align}
The optimizer is exactly \eqref{eq:tilted-reduced-kernel}, and it is measurable in $x$.  Integrating \eqref{eq:pointwise-Gibbs}, or equivalently applying the same formula to full kernels, yields
\begin{align}
 &\int_0^1\log\bigg(
 \sum_{k\ge1}\pi_k
 \exp\bigg\{\sum_{j=1}^m\lambda_j
 \1_{\{xk\le t_j\}}\bigg\}
 \bigg)\dd x
 \nonumber\\
 &\qquad=\sup_{\nu\in\mathscr K}
 \bigg\{
 \sum_{j=1}^m\lambda_j(T\nu)(t_j)
 -\mathcal E_\theta(\nu)
 \bigg\}
 \nonumber\\
 &\qquad=\sup_{z\in\R^m}
 \{\ip{\lambda}{z}-\mathcal J_F(z)\}.
 \label{eq:Gibbs-integrated}
\end{align}
Since $\mathcal J_F$ is proper, convex and lower semicontinuous, the Fenchel--Moreau theorem \cite[Theorem~12.2]{Rockafellar} gives
\begin{equation}\label{eq:finite-entropy-dual}
 \mathcal J_F(z)
 =\sup_{\lambda\in\R^m}
 \left\{
 \ip{\lambda}{z}-\Lambda_F(\lambda)
 \right\},
\end{equation}
where $\Lambda_F$ denotes the integral on the first line of \eqref{eq:Gibbs-integrated}.

Let $\widetilde I_\theta$ be the right-hand side of \eqref{eq:projective-rate}.  If $T\nu=f$, then \eqref{eq:finite-entropy-dual} implies
\[
 \widetilde I_\theta(f)\le\mathcal E_\theta(\nu).
\]
Taking the infimum gives
\begin{equation}\label{eq:one-side-entropy}
 \widetilde I_\theta(f)\le I_\theta(f).
\end{equation}

For the converse direction, assume $M:=\widetilde I_\theta(f)<\infty$ and let
\[
 K_M:=\{\nu\in\mathscr K:\mathcal E_\theta(\nu)\le M\}.
\]
For $t\in[0,1]$, set
\[
 C_t:=\{\nu\in K_M:(T\nu)(t)=f(t)\}.
\]
Each $C_t$ is closed in the compact set $K_M$.  If $F=\{t_1,\ldots,t_m\}$ is finite, then \eqref{eq:finite-entropy-dual} and the definition of $M$ give
\[
 \mathcal J_F(f(t_1),\ldots,f(t_m))\le M.
\]
Attainment in \eqref{eq:JF} therefore shows that $\bigcap_{t\in F}C_t$ is nonempty.  The finite intersection property yields
\[
 \bigcap_{t\in[0,1]}C_t\ne\varnothing.
\]
Thus some kernel $\nu$ satisfies $T\nu=f$ and $\mathcal E_\theta(\nu)\le M$, so $I_\theta(f)\le M$.  Together with \eqref{eq:one-side-entropy}, this proves \eqref{eq:projective-rate}.  Formula \eqref{eq:functional-dual-rate} is obtained by writing $\mu=\sum_{j=1}^m\lambda_j\delta_{t_j}$. Then
\[
 \sum_{j=1}^m\lambda_j\1_{\{xk\le t_j\}}
 =\mu([xk,1]).
\]
\end{proof}

\subsection{Properties of the functional LDP rate}

The entropy representation immediately implies the main structural properties of finite-rate paths.

\begin{proposition}\label{prop:LDP-rate-properties}
The rate function $I_\theta$ is convex.  If $I_\theta(f)<\infty$, then $f$ is absolutely continuous and nondecreasing, satisfies $f(0)=0$, and obeys
\begin{equation}\label{eq:path-bound}
 0\le f(t)\le t,
 \qquad 0\le t\le1.
\end{equation}
For every admissible representation $f=\mathcal T a$,
\begin{equation}\label{eq:path-derivative}
 f'(t)=\sum_{k\ge1}\frac1k a_k(t/k)
 \quad\text{for a.e. }t\in(0,1).
\end{equation}
The unique zero of $I_\theta$ is
\begin{equation}\label{eq:LLN-path}
 \bar f_\theta(t)=c_\theta t.
\end{equation}
Furthermore,
\begin{equation}\label{eq:zero-path-infinite}
 I_\theta(0)=+\infty.
\end{equation}
\end{proposition}

\begin{proof}
Convexity follows directly from the entropy contraction because $\mathcal E_\theta$ is convex and $T$ is linear. It also follows from the convex dual formula \eqref{eq:functional-dual-rate}.  If $I_\theta(f)<\infty$, Proposition~\ref{prop:entropy-equivalence} provides an admissible minimizer, so write $f=\mathcal T a$.  Define a finite measure $\eta_a$ on $[0,1]$ by
\[
 \eta_a(B)
 :=\sum_{k\ge1}\int_0^{1/k}
 \1_{\{kx\in B\}}a_k(x)\dd x.
\]
For each $k$, the change of variables $t=kx$ gives
\[
 \int_0^{1/k}\1_{\{kx\in B\}}a_k(x)\dd x
 =\int_B\frac1k a_k(t/k)\dd t.
\]
All summands are nonnegative, so Tonelli's theorem yields
\[
 \eta_a(B)=\int_B\sum_{k\ge1}\frac1k a_k(t/k)\dd t.
\]
Since $f(t)=\eta_a([0,t])$, the function $f$ is absolutely continuous and \eqref{eq:path-derivative} holds.  In particular, $f(0)=0$ and $f$ is nondecreasing.  Moreover, if $xk\le t\le1$, then $x\le t$ and $k\le m(x)$.  Hence
\begin{align*}
 f(t)
 =\int_0^1\sum_{k\ge1}
 \1_{\{xk\le t\}}a_k(x)\dd x
 \le\int_0^t\sum_{k=1}^{m(x)}a_k(x)\dd x
 \le t,
\end{align*}
which proves \eqref{eq:path-bound}.

By Proposition~\ref{prop:entropy-equivalence}, the entropy infimum is attained whenever it is finite.  The integrand in \eqref{eq:reduced-entropy} is the relative entropy of the probability vector
\[
 (a_0(x),a_1(x),\ldots,a_{m(x)}(x))
\]
with respect to
\[
 (q^{m(x)},\pi_1,\ldots,\pi_{m(x)}).
\]
It is nonnegative and vanishes if and only if these two vectors agree.  Therefore zero cost requires, for almost every $x$,
\[
 a_k(x)=\pi_k,\quad1\le k\le m(x),
 \qquad
 a_0(x)=q^{m(x)}.
\]
The corresponding path is
\[
 (\mathcal T a)(t)
 =\sum_{k\ge1}\pi_k\int_0^{t/k}\dd x
 =t\sum_{k\ge1}\frac{\pi_k}{k}
 =c_\theta t.
\]
This proves both existence and uniqueness of the zero in \eqref{eq:LLN-path}.

Finally, suppose $\mathcal T a=0$.  Evaluating at $t=1$ gives
\[
 0=(\mathcal T a)(1)
 =\sum_{k\ge1}\int_0^{1/k}a_k(x)\dd x.
\]
Every term is nonnegative, so $a_k(x)=0$ for almost every $x\le1/k$ and every $k$.  Equivalently, all visible coordinates vanish and $a_0(x)=1$ for almost every $x$.  Consequently,
\[
 \mathcal R_\theta(a)
 =-\log q\int_0^1m(x)\dd x.
\]
Since $m(x)=j$ on $(1/(j+1),1/j]$,
\[
 \int_0^1m(x)\dd x
 =\sum_{j\ge1}j\left(\frac1j-\frac1{j+1}\right)
 =\sum_{j\ge1}\frac1{j+1}=+\infty.
\]
Thus every kernel generating the zero path has infinite entropy, and \eqref{eq:zero-path-infinite} follows.
\end{proof}

\begin{remark}\label{rem:D0-compactness}
Let
\begin{equation*}
 \begin{split}
 D_0:=\bigl\{f\in D[0,1]:\;& f(0)=0,\ f\text{ is nondecreasing},\\
 &0\le f(t)\le t\text{ for every }t\in[0,1]\bigr\}.
 \end{split}
\end{equation*}
Proposition~\ref{prop:LDP-rate-properties} shows that every finite-rate path belongs to $D_0$ and is, in fact, absolutely continuous.  The set $D_0$ should not, however, be used as a compact containment set for the $J_1$-topology.  In fact, bounded monotone subsets of $D[0,1]$ need not be relatively compact in $J_1$-topology, because two nearby jumps cannot merge into a single jump under a $J_1$ time change.

For a direct counterexample, let $n\ge4$ and define
\[
 f_n(t):=
 \begin{cases}
 0,&0\le t<\frac12-\frac1n,\\
 \frac14,&\frac12-\frac1n\le t<\frac12+\frac1n,\\
 \frac12,&\frac12+\frac1n\le t\le1.
 \end{cases}
\]
Then $f_n\in D_0$.  For every fixed $\delta>0$ and all sufficiently large $n$, the two jumps are separated by less than $\delta$.  The $J_1$ oscillation modulus in the compactness criterion is therefore at least $1/4$: an admissible partition whose consecutive points are more than $\delta$ apart cannot isolate both jumps.  Hence $(f_n)$ has no $J_1$-convergent subsequence.  This also explains why the proof of Theorem~\ref{thm:FLDP} uses the explicit exponential-equicontinuity estimate \eqref{eq:LDP-exponential-equicontinuity}, rather than deterministic confinement to a compact monotone class.

The finite $n$ paths nevertheless approach $D_0$ deterministically.  Put
\[
 \varepsilon_n:=\sup_{1\le m\le n}\frac{(\psi(m)-m)_+}{n}.
\]
We first show that $\varepsilon_n\to0$.  Fix $\delta\in(0,1)$.  If $m\le\delta n$, the Chebyshev bound $\psi(m)\le Cm$ gives
$(\psi(m)-m)_+/n\le C\delta$.  If $m>\delta n$, the prime number theorem gives
\[
 \sup_{\delta n<m\le n}\left|\frac{\psi(m)}m-1\right|\longrightarrow0.
\]
Letting $\delta\downarrow0$ proves the claim.  Since
\[
 X_n(t)\le\frac{\psi(\lfloor nt\rfloor)}n\le t+\varepsilon_n,
\]
the function $g_n(t):=\min\{X_n(t),t\}$ belongs to $D_0$ and satisfies
\begin{equation}\label{eq:dist-D0}
 \inf_{g\in D_0}\|X_n-g\|_\infty
 \le\|X_n-g_n\|_\infty\le\varepsilon_n\longrightarrow0.
\end{equation}
Thus $D_0$ describes the natural effective domain constraint, but its proximity to $X_n$ is not a substitute for the $J_1$ tightness argument.
\end{remark}

\subsection{Endpoint rate and its local quadratic expansion}

We give the analytic properties of the endpoint rate, including its boundary behavior and local quadratic expansion.

\begin{proposition}\label{prop:endpoint-rate}
The function $\Lambda_\theta$ in \eqref{eq:endpoint-cumulant} is finite and real analytic on $\R$, with
\begin{align}
 \Lambda_\theta'(\lambda)
 &=\sum_{j\ge1}\frac1{j(j+1)}
 \frac{(1-q^j)e^\lambda}
 {q^j+(1-q^j)e^\lambda},
 \label{eq:endpoint-first-derivative}\\
 \Lambda_\theta''(\lambda)
 &=\sum_{j\ge1}\frac1{j(j+1)}
 \frac{q^j(1-q^j)e^\lambda}
 {\bigl(q^j+(1-q^j)e^\lambda\bigr)^2}>0.
 \label{eq:endpoint-second-derivative}
\end{align}
Moreover,
\begin{equation}\label{eq:endpoint-derivative-range}
 \lim_{\lambda\to-\infty}\Lambda_\theta'(\lambda)=0,
 \qquad
 \lim_{\lambda\to\infty}\Lambda_\theta'(\lambda)=1.
\end{equation}
Consequently, the equation \eqref{eq:lambda-y-equation} has a unique solution $\lambda_y$ for every $y\in(0,1)$, the endpoint rate is strictly convex on $(0,1)$, and the formulas in Corollary~\ref{cor:endpoint-LDP} hold.  In addition,
\begin{equation}\label{eq:endpoint-local-quadratic}
 I_\theta^{\mathrm{end}}(c_\theta+u)
 =\frac{u^2}{2\sigma_\theta^2}+O(u^3),
 \qquad u\to0.
\end{equation}
\end{proposition}

\begin{proof}
For each $j$, the summand in \eqref{eq:endpoint-cumulant} is $1/[j(j+1)]$ times the logarithmic moment generating function of a Bernoulli random variable with success probability $1-q^j$.  On every compact set of $\lambda$-values, derivatives of any fixed order are uniformly bounded in $j$.  Since
\[
 \sum_{j\ge1}\frac1{j(j+1)}=1,
\]
termwise differentiation is justified by the Weierstrass test, and \eqref{eq:endpoint-first-derivative}--\eqref{eq:endpoint-second-derivative} follow.  The strict positivity in \eqref{eq:endpoint-second-derivative} proves strict convexity of $\Lambda_\theta$.  Dominated convergence gives \eqref{eq:endpoint-derivative-range}.  Hence $\Lambda_\theta'$ is a strictly increasing bijection from $\R$ onto $(0,1)$, so the Legendre transform is attained at the unique $\lambda_y$ when $0<y<1$.  Its second derivative on this interval is
\[
 \bigl(I_\theta^{\mathrm{end}}\bigr)''(y)
 =\frac1{\Lambda_\theta''(\lambda_y)}>0,
\]
which proves strict convexity of the rate.  Formula
\eqref{eq:endpoint-first-derivative} shows that the function $u_y$ in
\eqref{eq:endpoint-optimal-profile} satisfies $\int_0^1u_y(x)\dd x=y$.
For $r(x):=1-q^{m(x)}$, direct substitution of
$u_y(x)=r(x)e^{\lambda_y}/(1-r(x)+r(x)e^{\lambda_y})$ gives
\begin{align*}
 &u_y(x)\log\frac{u_y(x)}{r(x)}
 +(1-u_y(x))\log\frac{1-u_y(x)}{1-r(x)}\\
 &\hspace{35mm}=\lambda_yu_y(x)
 -\log\bigl(1-r(x)+r(x)e^{\lambda_y}\bigr).
\end{align*}
Integration proves \eqref{eq:endpoint-binary-entropy}.

For $y=1$, rewrite
\[
 \Lambda_\theta(\lambda)-\lambda
 =\sum_{j\ge1}\frac1{j(j+1)}
 \log\left((1-q^j)+q^je^{-\lambda}\right).
\]
The summands are bounded above by zero and, for $\lambda\ge0$, bounded below by
$\log(1-q^j)/[j(j+1)]$.  The latter series is summable because
$-\log(1-q^j)\sim q^j$.  Dominated convergence therefore yields
\[
 \lim_{\lambda\to\infty}
 \bigl(\Lambda_\theta(\lambda)-\lambda\bigr)
 =\sum_{j\ge1}\frac1{j(j+1)}\log(1-q^j),
\]
which proves the value at $y=1$.

To determine the value at zero, fix $J\ge1$ and take
$\lambda\le J\log q<0$.  Then $e^\lambda\le q^J\le q^j$ for $1\le j\le J$, and hence
\[
 q^j+(1-q^j)e^\lambda\le2q^j,
 \qquad 1\le j\le J.
\]
For $j>J$ the logarithm in \eqref{eq:endpoint-cumulant} is nonpositive.  It follows that
\[
 \Lambda_\theta(\lambda)
 \le \log2\sum_{j=1}^J\frac1{j(j+1)}
 +\log q\sum_{j=1}^J\frac1{j+1}.
\]
The right-hand side tends to $-\infty$ as $J\to\infty$.  Therefore
$\Lambda_\theta(\lambda)\to-\infty$ as $\lambda\to-\infty$, and
\[
 I_\theta^{\mathrm{end}}(0)
 =\sup_{\lambda\in\R}\{-\Lambda_\theta(\lambda)\}=+\infty.
\]
The asymptotic slopes in \eqref{eq:endpoint-derivative-range} imply that the Legendre transform is infinite outside $[0,1]$.

At the origin,
\begin{align}
 \Lambda_\theta'(0)
 &=\sum_{j\ge1}\frac{1-q^j}{j(j+1)}
 =\sum_{k\ge1}\pi_k\sum_{j\ge k}\frac1{j(j+1)}
 =\sum_{k\ge1}\frac{\pi_k}{k}
 =c_\theta,
 \label{eq:endpoint-mean}\\
 \Lambda_\theta''(0)
 &=\sum_{j\ge1}\frac{q^j(1-q^j)}{j(j+1)}
 =\sigma_\theta^2.
 \label{eq:endpoint-variance}
\end{align}
The third derivative of $\Lambda_\theta$ is bounded on a neighborhood of zero.  By the inverse function theorem,
\[
 \lambda_{c_\theta+u}
 =\frac{u}{\sigma_\theta^2}+O(u^2).
\]
Substitution into
$I_\theta^{\mathrm{end}}(c_\theta+u)
=(c_\theta+u)\lambda_{c_\theta+u}
-\Lambda_\theta(\lambda_{c_\theta+u})$
gives \eqref{eq:endpoint-local-quadratic}.
\end{proof}

\subsection{A variational form of the MDP rate}

Let $\mathscr L_\theta$ be the space of equivalence classes of measurable arrays $h=(h_k(x))_{k\ge1}$ satisfying
\begin{equation}\label{eq:Ltheta-space}
 \sum_{k\ge1}h_k(x)=0 \quad\text{for a.e. }x,
 \qquad
 \norm{h}_{\mathscr L_\theta}^2
 :=\int_0^1\sum_{k\ge1}\frac{h^2_k(x)}{\pi_k}\dd x<\infty.
\end{equation}
Indeed, $\mathscr L_\theta$ is a closed subspace of the ambient weighted $L^2$ space.  To see this, note that
\begin{equation}\label{eq:Ltheta-l1-control}
 \int_0^1\sum_{k\ge1}|h_k(x)|\dd x
 \le
 \int_0^1\left(\sum_{k\ge1}\frac{h^2_k(x)}{\pi_k}\right)^{1/2}\dd x
 \le \norm{h}_{\mathscr L_\theta}.
\end{equation}
Hence $h\mapsto \sum_{k\ge1}h_k(\cdot)$ is a continuous map from the ambient weighted $L^2$ space into $L^1[0,1]$, and its kernel is closed.  Formula \eqref{eq:Ltheta-l1-control} also shows that, for every $t\in[0,1]$, the integral
\begin{equation}\label{eq:A-operator}
 (Ah)(t) := \int_0^1\sum_{k\ge1}\1_{\{xk\le t\}}h_k(x)\dd x
\end{equation}
is absolutely convergent and defines a continuous linear functional of $h$.

The following Hilbert space contraction identifies the quadratic MDP rate with the RKHS norm.

\begin{proposition}\label{prop:quadratic-contraction}
For every $f\in C[0,1]$,
\begin{equation}\label{eq:quadratic-contraction}
 J_\theta(f)
 =\frac12\inf\left\{
 \norm{h}_{\mathscr L_\theta}^2:
 h\in\mathscr L_\theta,\ Ah=f
 \right\}.
\end{equation}
\end{proposition}

\begin{proof}
The preceding paragraph already gives the absolute convergence needed for $A$.  The continuity of the functions $Ah$ follows in the proof from the continuity of $t\mapsto e_t$ in $\mathscr L_\theta$.

For $t\in[0,1]$, define
\begin{equation}\label{eq:e-t}
 e_t(x,k)
 :=\pi_k\left(
 \1_{\{xk\le t\}}
 -\sum_{l\ge1}\pi_l\1_{\{xl\le t\}}
 \right).
\end{equation}
Then $e_t\in\mathscr L_\theta$ and $\sum_ke_t(x,k)=0$.  Moreover,
\[
 \norm{e_t-e_s}_{\mathscr L_\theta}^2
 =C_\theta(t,t)+C_\theta(s,s)-2C_\theta(s,t),
\]
which tends to zero as $t\to s$ by dominated convergence in
\eqref{eq:cov-integral}.  A direct calculation gives
\begin{align*}
 \ip{e_s}{e_t}_{\mathscr L_\theta}
 =\int_0^1
 \Cov\bigl(\1_{\{xG\le s\}},\1_{\{xG\le t\}}\bigr)\dd x=C_\theta(s,t).
\end{align*}
Furthermore, since $\sum_kh_k(x)=0$,
\begin{align*}
 \ip{h}{e_t}_{\mathscr L_\theta}
 &=\int_0^1\sum_{k\ge1}h_k(x)
 \1_{\{xk\le t\}}\dd x
 =(Ah)(t).
\end{align*}

Let $\mathscr L_0$ be the closed linear span of $\{e_t:0\le t\le1\}$ in
$\mathscr L_\theta$.  The function $Ah$ depends only on the orthogonal projection of $h$ onto $\mathscr L_0$.  The map
\[
 A:\mathscr L_0\longrightarrow\R^{[0,1]}
\]
is injective, because $Ah=0$ implies $\ip{h}{e_t}=0$ for every $t$ and hence $h=0$ in $\mathscr L_0$.  Endow its range with the transported Hilbert norm
\[
 \norm{Ah}_{A}:=\norm{h}_{\mathscr L_\theta},
 \qquad h\in\mathscr L_0.
\]
For every $t$,
\[
 A e_t(s)=\ip{e_t}{e_s}_{\mathscr L_\theta}=C_\theta(s,t),
\]
so the range contains all kernel sections.  Moreover,
\[
 (Ah)(t)=\ip{h}{e_t}_{\mathscr L_\theta}
 =\ip{Ah}{C_\theta(\cdot,t)}_{A}.
\]
Therefore the range of $A$, with the norm $\norm{\cdot}_A$, is exactly the reproducing kernel Hilbert space $\cH_\theta$ associated with $C_\theta$.

For arbitrary $h\in\mathscr L_\theta$, replacing $h$ by its projection onto
$\mathscr L_0$ leaves $Ah$ unchanged and can only decrease the norm.  Hence
\[
 \norm{f}_{\cH_\theta}^2
 =\inf_{Ah=f}\norm{h}_{\mathscr L_\theta}^2,
\]
with the convention that the infimum is infinite if $f$ is not in the range.  Combining this identity with \eqref{eq:MDP-rate} proves \eqref{eq:quadratic-contraction}.
\end{proof}

\section{Arithmetic reduction and geometric marks}\label{sec:reduction}

Let $\Pp$ be the set of prime numbers.  The von Mangoldt function is
$\Lambda(p^r)=\log p$ for prime powers $p^r$ with $r\ge1$ and
$\Lambda(m)=0$ otherwise.  See \cite[Chapter~4]{Apostol}.  We write
\[
 \vartheta(x):=\sum_{p\in\Pp,\;p\le x}\log p,
 \qquad
 \psi(x):=\sum_{m\le x}\Lambda(m),
\]
for the first and second Chebyshev functions, respectively.
For a finite set $A\subset\N$, put
\[
 I_A(m):=\1_{\{A\cap m\N\ne\varnothing\}}.
\]
Then
\begin{equation}\label{eq:von-Mangoldt-LCM}
 \log\lcm(A)=\sum_{m\ge1}\Lambda(m)I_A(m),
\end{equation}
an elementary identity also stated in \cite[Lemma~2.1]{AKM}.  We use the classical Chebyshev bound $\vartheta(x)\le Cx$ and the prime number theorem throughout. Standard references include \cite[Chapter~4]{Apostol}.

Define the large prime process
\begin{equation}\label{eq:Sn}
 S_n(t):=\sum_{\substack{p\in\Pp\\\sqrt n<p\le n}}
 \log p\,I_{A_{\floor{nt}}}(p),
 \qquad 0\le t\le1,
\end{equation}
which is the large prime approximation to $\log L_{\floor{nt}}$.
The basic arithmetic approximation is the following deterministic uniform estimate.

\begin{lemma}\label{lem:deterministic-approximation}
There is a constant $C<\infty$ such that, for all sufficiently large $n$,
\begin{equation}\label{eq:deterministic-approximation}
 \sup_{0\le t\le1}
 \abs{\log L_{\floor{nt}}-S_n(t)}
 \le C\sqrt n.
\end{equation}
\end{lemma}

\begin{proof}
By \eqref{eq:von-Mangoldt-LCM}, the difference in \eqref{eq:deterministic-approximation} is nonnegative.  Uniformly in $t$ it is bounded by the total possible contribution of primes $p\le\sqrt n$ and of prime powers $p^r$ with $r\ge2$:
\begin{align*}
 0\le\log L_{\floor{nt}}-S_n(t)
 &\le\sum_{p\le\sqrt n}\log p
 +\sum_{r\ge2}\sum_{p^r\le n}\log p\\
 &=\vartheta(\sqrt n)
 +\sum_{2\le r\le\log_2n}\vartheta(n^{1/r}).
\end{align*}
The Chebyshev bound $\vartheta(x)\le C_0x$ yields
\begin{align*}
 \vartheta(\sqrt n)
 +\sum_{2\le r\le\log_2n}\vartheta(n^{1/r})
 &\le2C_0\sqrt n
 +C_0n^{1/3}\log_2n\\
 &\le C\sqrt n
\end{align*}
for all sufficiently large $n$.
\end{proof}

For a prime $p$ with $\sqrt n<p\le n$, the indicator process $I_{A_{\floor{nt}}}(p)$ depends only on the Bernoulli variables
\[
 \xi_p,\xi_{2p},\ldots,\xi_{\floor{n/p}p}.
\]
If $p$ and $q$ are distinct primes exceeding $\sqrt n$, these two finite blocks are disjoint: a common index not exceeding $n$ would be divisible by $pq>n$.  Hence the blocks indexed by such primes are mutually independent.  To avoid any overlap beyond the observation range, enlarge the probability space as follows.  For every prime $p$ with $\sqrt n<p\le n$, let
$(\zeta_{n,p,k})_{k\ge1}$ be a Bernoulli sequence such that
\[
 \zeta_{n,p,k}=\xi_{kp},
 \qquad 1\le k\le\floor{n/p},
\]
and extend this finite block by fresh Bernoulli variables of parameter $\theta$.  The extensions are chosen independently for different primes and independently of the original variables.  Put
\begin{equation}\label{eq:geometric-mark}
 G_{n,p}:=\min\{k\ge1:\zeta_{n,p,k}=1\}.
\end{equation}
The enlargement does not change any variable determined by indices at most $n$.  The marks $(G_{n,p})_{\sqrt n<p\le n}$ are independent and satisfy
$\Prob(G_{n,p}=k)=\pi_k$.  Since $pG_{n,p}$ is integer valued,
\begin{equation}\label{eq:indicator-geometric}
 I_{A_{\floor{nt}}}(p)
 =\1_{\{pG_{n,p}\le\floor{nt}\}}
 =\1_{\{pG_{n,p}\le nt\}}.
\end{equation}
Thus
\begin{equation}\label{eq:Sn-geometric}
 S_n(t)=\sum_{\sqrt n<p\le n}\log p\,
 \1_{\{(p/n)G_{n,p}\le t\}}.
\end{equation}

Set
\begin{equation}\label{eq:x-r-v}
 x_{n,p}:=\frac pn,
 \qquad
 r_{n,p}:=\frac{\log p}{\log n},
 \qquad
 v_n:=\frac n{\log n},
\end{equation}
and define
\begin{equation}\label{eq:Wn}
 W_n(t):=\frac{S_n(t)}n
 =\frac1{v_n}\sum_{\sqrt n<p\le n}
 r_{n,p}\1_{\{x_{n,p}G_{n,p}\le t\}}.
\end{equation}
Lemma~\ref{lem:deterministic-approximation} gives
\begin{equation}\label{eq:LDP-deterministic-equivalence}
 \norm{X_n-W_n}_\infty\le\frac C{\sqrt n}.
\end{equation}
After centering,
\begin{align}
 \sup_{0\le t\le1}
 \bigg|Y_n(t)
 -\frac{S_n(t)-\E S_n(t)}
 {b_n\sqrt{n\log n}}\bigg|
 &\le\frac{2C}{b_n\sqrt{\log n}}.
 \label{eq:MDP-deterministic-equivalence}
\end{align}

Finally, since
\[
 0\le\log L_m-\log L_{m-1}\le\log m\le\log n,
\]
the polygonal interpolations satisfy
\begin{align}
 \norm{X_n-\widehat X_n}_\infty
 &\le\frac{\log n}{n},
 \label{eq:LDP-interpolation}\\
 \norm{Y_n-\widehat Y_n}_\infty
 &\le\frac{2\log n}{b_n\sqrt{n\log n}}
 =\frac2{b_n}\sqrt{\frac{\log n}{n}}.
 \label{eq:MDP-interpolation}
\end{align}
All these errors converge to zero deterministically and hence are exponentially negligible at the corresponding speeds.

\section{Prime Riemann sums and covariance estimates}\label{sec:prime}

Let $\pi(x):=\#\{p\in\Pp: p\le x\}$ be the prime counting function.  We use the prime number theorem in the form
\begin{equation}\label{eq:PNT}
 \pi(x)\sim\frac{x}{\log x},
 \qquad x\to\infty.
\end{equation}

The large prime sums will repeatedly be converted into integrals by the next Riemann sum estimate.

\begin{lemma}\label{lem:prime-Riemann}
Let $h:[0,1]\to\R$ be bounded and Riemann integrable.  Then
\begin{equation}\label{eq:prime-Riemann-basic}
 \frac1{v_n}\sum_{\sqrt n<p\le n}h(p/n)
 \longrightarrow\int_0^1h(x)\dd x.
\end{equation}
If $H:[0,1]\times[0,1]\to\R$ is bounded, is Riemann integrable in its first variable after setting the second variable equal to $1$, and is uniformly continuous in the second variable on $[\delta,1]\times[0,1]$ for every $\delta>0$, then
\begin{equation}\label{eq:prime-Riemann-two}
 \frac1{v_n}\sum_{\sqrt n<p\le n}
 H\left(\frac pn,\frac{\log p}{\log n}\right)
 \longrightarrow\int_0^1H(x,1)\dd x.
\end{equation}
In particular, for every fixed $r\ge0$,
\begin{equation}\label{eq:prime-Riemann-weighted}
 \frac1{v_n}\sum_{\sqrt n<p\le n}
 r_{n,p}^{\,r}h(p/n)
 \longrightarrow\int_0^1h(x)\dd x.
\end{equation}
\end{lemma}

\begin{proof}
For $0<a<b\le1$, the prime number theorem gives
\begin{align*}
 \frac1{v_n}\#\{p:na<p\le nb\}
 &=\frac{\log n}{n}\bigl(\pi(nb)-\pi(na)\bigr)\\
 &\longrightarrow b-a.
\end{align*}
Hence \eqref{eq:prime-Riemann-basic} holds for step functions supported in $[\delta,1]$.  Upper and lower step approximations prove it for Riemann integrable functions on $[\delta,1]$.  The contribution of $(0,\delta]$ is bounded by
\[
 \norm{h}_\infty\frac{\log n}{n}\pi(n\delta),
\]
whose limsup is at most $\norm{h}_\infty\delta$.  Letting $\delta\downarrow0$ proves \eqref{eq:prime-Riemann-basic}.

Uniformly for $p/n\in[\delta,1]$,
\[
 \frac{\log p}{\log n}
 =1+\frac{\log(p/n)}{\log n}
 \longrightarrow1.
\]
The asserted uniform continuity in the second variable therefore reduces \eqref{eq:prime-Riemann-two} on $[\delta,1]$ to \eqref{eq:prime-Riemann-basic}.  Boundedness controls the interval $(0,\delta]$, and then $\delta\downarrow0$.  Formula \eqref{eq:prime-Riemann-weighted} is obtained by taking $H(x,u)=u^rh(x)$.
\end{proof}

The preceding Riemann sum estimate gives the following mean and covariance limits.

\begin{lemma}\label{lem:mean-covariance}
For $0\le s\le t\le1$,
\begin{align}
 \E[W_n(t)-W_n(s)]
 &\longrightarrow
 \int_0^1\Prob(s<xG\le t)\dd x
 =c_\theta(t-s).
 \label{eq:mean-limit}
\end{align}
For $0\le s,t\le1$,
\begin{align}\label{eq:covariance-sum-limit}
 &\frac1{n\log n}
 \sum_{\sqrt n<p\le n}(\log p)^2
 \Cov\Big(
 \1_{\{x_{n,p}G_{n,p}\le s\}},
 \1_{\{x_{n,p}G_{n,p}\le t\}}
 \Big)
\longrightarrow C_\theta(s,t).
\end{align}
\end{lemma}

\begin{proof}
For fixed $s,t$, the functions of $x$ appearing in the expectations and covariances are bounded and have discontinuities only at the countable set
$\{s/k,t/k:k\ge1\}$, so they are Riemann integrable.  Lemma~\ref{lem:prime-Riemann}, with one or two powers of $r_{n,p}$ as appropriate, gives both limits.

The integral in \eqref{eq:mean-limit} can be evaluated exactly.  Since $0\le s\le t\le1$ and $k\ge1$,
\[
 \operatorname{Leb}\{x\in[0,1]:s<xk\le t\}
 =\frac{t-s}{k}.
\]
Tonelli's theorem therefore gives
\[
 \int_0^1\Prob(s<xG\le t)\dd x
 =\sum_{k\ge1}\pi_k\frac{t-s}{k}
 =c_\theta(t-s).
\]
\end{proof}

We also need a uniform local version of the variance estimate.

\begin{lemma}\label{lem:local-variance}
For $\delta>0$, set
\begin{equation}\label{eq:local-variance-definition}
 V_n(\delta)
 :=\sup_{\substack{0\le s\le t\le1\\t-s\le\delta}}
 \frac1{n\log n}
 \sum_{\sqrt n<p\le n}(\log p)^2
 \Prob(s<x_{n,p}G_{n,p}\le t).
\end{equation}
Then
\begin{equation}\label{eq:local-variance-bound}
 \limsup_{n\to\infty}V_n(\delta)
 \le2c_\theta\delta,
 \qquad 0<\delta\le1.
\end{equation}
\end{lemma}

\begin{proof}
Fix $\delta\in(0,1]$.  Let $M:=\lceil1/\delta\rceil$ and use the regular grid
\[
 r_j:=(j\delta)\wedge1,
 \qquad 0\le j\le M.
\]
Every interval $(s,t]\subset[0,1]$ of length at most $\delta$ is contained in the union of at most two adjacent grid cells, and the containing union has length at most $2\delta$.  Only finitely many such unions occur for fixed $\delta$.  For a fixed union $(u,v]$, Lemma~\ref{lem:prime-Riemann} with the factor $r_{n,p}^2$ gives
\begin{align*}
 &\frac1{n\log n}
 \sum_{\sqrt n<p\le n}(\log p)^2
 \Prob(u<x_{n,p}G_{n,p}\le v)\\
 &\qquad=\frac1{v_n}
 \sum_{\sqrt n<p\le n}r_{n,p}^2
 \Prob(u<x_{n,p}G_{n,p}\le v)\\
 &\qquad\longrightarrow
 \int_0^1\Prob(u<xG\le v)\dd x
 =c_\theta(v-u)
 \le2c_\theta\delta.
\end{align*}
Taking the maximum over the finite collection of adjacent-cell unions and then the limsup proves \eqref{eq:local-variance-bound}.
\end{proof}

\section{Proof of the functional LDP}\label{sec:LDP}

By \eqref{eq:LDP-deterministic-equivalence} and \eqref{eq:LDP-interpolation}, it suffices to prove the LDP for the polygonal interpolation $\widehat W_n$ of $W_n$.

\medskip
\subsection{Finite-dimensional large deviations}

Fix $m\ge1$, $0\le t_1<\cdots<t_m\le1$, and
$\lambda=(\lambda_1,\ldots,\lambda_m)\in\R^m$.  Independence of the geometric marks gives
\begin{align}
 &\frac1{v_n}
 \log\E\exp\bigg\{
 v_n\sum_{j=1}^m\lambda_jW_n(t_j)
 \bigg\}
 \nonumber\\
 &\quad=\frac1{v_n}\sum_{\sqrt n<p\le n}
 \log\bigg[
 \sum_{k\ge1}\pi_k
 \exp\bigg\{
 r_{n,p}\sum_{j=1}^m\lambda_j
 \1_{\{x_{n,p}k\le t_j\}}
 \bigg\}
 \bigg].
 \label{eq:LDP-fdd-mgf}
\end{align}
For fixed $\lambda$, the logarithms in \eqref{eq:LDP-fdd-mgf} are uniformly bounded.  On every interval $x\in[\varepsilon,1]$, their dependence on the second argument $r_{n,p}$ is uniformly continuous, and the limiting dependence on $x$ has discontinuities only at the points $t_j/k$.  Lemma~\ref{lem:prime-Riemann} therefore yields
\begin{align}
 &\lim_{n\to\infty}\frac1{v_n}
 \log\E\exp\bigg\{
 v_n\sum_{j=1}^m\lambda_jW_n(t_j)
 \bigg\}
 \nonumber\\
 &\quad=\Lambda_{\boldsymbol t}(\lambda)
 :=\int_0^1\log\bigg[
 q^{m(x)}+
 \sum_{k=1}^{m(x)}\pi_k
 \exp\bigg\{\sum_{j=1}^m\lambda_j
 \1_{\{xk\le t_j\}}\bigg\}
 \bigg]\dd x.
 \label{eq:LDP-fdd-limit}
\end{align}
The expression in square brackets is the same as the infinite sum over $k\ge1$, because all invisible marks $k>m(x)$ have exponent zero and total mass $q^{m(x)}$.

The function $\Lambda_{\boldsymbol t}$ is finite on all of $\R^m$.  Differentiation under the integral is justified by boundedness of the indicator vectors and gives derivatives of every order.  Thus the effective domain is $\R^m$ and the function is essentially smooth.  The random vectors are uniformly bounded, so exponential tightness in $\R^m$ is automatic.  The G\"artner--Ellis theorem \cite[Theorem~2.3.6]{DZ} gives an LDP for
\[
 (W_n(t_1),\ldots,W_n(t_m))
\]
with speed $v_n$ and good convex rate
\begin{equation}\label{eq:LDP-fdd-rate}
 I_{\boldsymbol t}(z)
 =\sup_{\lambda\in\R^m}
 \{\ip{\lambda}{z}-\Lambda_{\boldsymbol t}(\lambda)\}.
\end{equation}
At every time the interpolation error is at most the largest jump, which is bounded by $1/v_n$.  Hence the same finite-dimensional LDP holds for $\widehat W_n$ by deterministic exponential equivalence.

\medskip
\subsection{Exponential equicontinuity}

The large deviation process is not centered and therefore preserves the monotonicity of the least common multiple path.  Monotonicity reduces the modulus estimate to nonnegative increments, but it does not by itself give compact containment in the $J_1$-topology.  See Remark~\ref{rem:D0-compactness}.  We therefore verify exponential equicontinuity explicitly.  The moderate deviation process in Section~\ref{sec:MDP} is centered and is no longer monotone, so its increment estimate requires a separate centered maximal inequality.

For $0\le s<t\le1$, put
\[
 D_n(s,t):=W_n(t)-W_n(s).
\]
This is a sum of independent nonnegative random variables.  For $0\le r\le1$ and $\lambda\ge0$, convexity of the exponential function gives
\[
 e^{\lambda r}-1\le r(e^\lambda-1).
\]
Consequently,
\begin{align}
 \frac1{v_n}\log\E e^{\lambda v_nD_n(s,t)}
 &\le\frac1{v_n}\sum_{\sqrt n<p\le n}
 \Prob(s<x_{n,p}G_{n,p}\le t)
 (e^{\lambda r_{n,p}}-1)
 \nonumber\\
 &\le(e^\lambda-1)\E D_n(s,t).
 \label{eq:LDP-increment-mgf}
\end{align}

Fix $\delta\in(0,1]$ and use the regular grid with mesh $\delta$.  Every interval of length at most $\delta$ is contained in the union of at most two adjacent grid cells.  There are at most $2\delta^{-1}+2$ such unions, and each has length at most $2\delta$.  For each fixed union, Lemma~\ref{lem:mean-covariance} gives convergence of the mean increment to $c_\theta$ times its length.  Since the collection is finite for fixed $\delta$, for all sufficiently large $n$ every such mean is at most $3c_\theta\delta$.

Assume $3c_\theta\delta<\varepsilon$.  Chernoff's inequality and \eqref{eq:LDP-increment-mgf} give, for every containing union $(s,t]$,
\begin{equation}\label{eq:LDP-increment-tail}
 \Prob(D_n(s,t)>\varepsilon)
 \le\exp\{-v_n\Psi(\varepsilon,3c_\theta\delta)\},
\end{equation}
where
\begin{equation}\label{eq:Psi}
 \Psi(\varepsilon,a)
 :=\sup_{\lambda\ge0}
 \{\lambda\varepsilon-a(e^\lambda-1)\}
 =\varepsilon\log\frac{\varepsilon}{a}-\varepsilon+a.
\end{equation}
The last identity follows by choosing $e^\lambda=\varepsilon/a$.  Since
$\Psi(\varepsilon,a)\to\infty$ as $a\downarrow0$, the finite union bound yields
\begin{equation}\label{eq:LDP-exponential-equicontinuity}
 \lim_{\delta\downarrow0}\limsup_{n\to\infty}
 \frac1{v_n}\log\Prob\bigg(
 \sup_{\abs{t-s}\le\delta}
 \abs{W_n(t)-W_n(s)}>\varepsilon
 \bigg)
 =-\infty.
\end{equation}
Here the absolute value is unnecessary because $W_n$ is nondecreasing.

At a fixed time, at most one prime contribution can jump.  Indeed, if
$pG_{n,p}=qG_{n,q}\le n$ for distinct primes $p,q>\sqrt n$, the common integer would be divisible by $pq>n$.  Hence every jump of $W_n$ is at most
$\max_pr_{n,p}/v_n\le1/v_n$.  In particular,
$\norm{\widehat W_n-W_n}_\infty\le1/v_n$, and therefore
\eqref{eq:LDP-exponential-equicontinuity} also holds for $\widehat W_n$.  In addition,
\[
 \sup_{0\le t\le1}\widehat W_n(t)
 \le\frac{\vartheta(n)}n,
\]
which is uniformly bounded by the Chebyshev estimate.

\medskip
\subsection{Completion of the LDP proof}

The Euclidean index set $[0,1]$ is totally bounded.  The finite-dimensional LDPs and \eqref{eq:LDP-exponential-equicontinuity} verify conditions (a.1)--(a.3) of \cite[Theorem~2.4]{ArconesLD}, with the inverse speed parameter equal to $v_n^{-1}$.  Therefore $\widehat W_n$ satisfies an LDP in $\ell^\infty[0,1]$ with good projective rate
\begin{equation}\label{eq:LDP-projective-proof}
 \widetilde I_\theta(f)
 =\sup_{m,\boldsymbol t}
 I_{\boldsymbol t}(f(t_1),\ldots,f(t_m)).
\end{equation}
The processes are measurable and separable because each is polygonal on the deterministic grid $\{j/n:0\le j\le n\}$.  They take values in the closed subspace $C[0,1]$ of $\ell^\infty[0,1]$, so the same LDP holds in $C[0,1]$ with the uniform norm.  Propositions~\ref{prop:entropy-equivalence} and~\ref{prop:dual-rate} identify the projective rate in \eqref{eq:LDP-projective-proof} with $I_\theta$ in \eqref{eq:functional-rate-main}.

The deterministic bounds \eqref{eq:LDP-deterministic-equivalence} and
\eqref{eq:LDP-interpolation} imply exponential equivalence at speed $v_n$.  See
\cite[Theorem~4.2.13]{DZ}.  They transfer the LDP from $\widehat W_n$ to
$\widehat X_n$.  The inclusion
\[
 (C[0,1],\norm{\cdot}_\infty)
 \longrightarrow(D[0,1],J_1)
\]
is continuous.  The contraction principle therefore gives an LDP for $\widehat X_n$ viewed as a $D[0,1]$-valued random element.  The rate is finite only at those $D$-paths that coincide with a continuous function. Equivalently, it is the same rate on $C[0,1]$ and is $+\infty$ on $D[0,1]\setminus C[0,1]$.  Finally,
$\norm{X_n-\widehat X_n}_\infty\to0$ deterministically, so exponential equivalence gives the claimed $J_1$-LDP for $X_n$.  This proves Theorem~\ref{thm:FLDP}.

The endpoint evaluation $f\mapsto f(1)$ is continuous both on $C[0,1]$ and on
$D[0,1]$ with the $J_1$ topology.  The contraction principle
\cite[Theorem~4.2.1]{DZ} therefore gives an endpoint LDP with rate
$\inf\{I_\theta(f):f(1)=y\}$.  On the other hand, the preceding finite-dimensional calculation with $m=1$ and $t_1=1$ gives directly the scalar Legendre rate in
\eqref{eq:endpoint-rate-Legendre}.  Uniqueness of the good rate function identifies these two expressions.  As for the explicit cumulant, if
\[
 x\in\left(\frac1{j+1},\frac1j\right],
\]
then $m(x)=j$ and $\Prob(xG\le1)=1-q^j$.  Hence
\begin{align*}
 \int_0^1\log\E e^{\lambda\1_{\{xG\le1\}}}\dd x
 &=\sum_{j\ge1}\frac1{j(j+1)}
 \log\bigl(q^j+(1-q^j)e^\lambda\bigr),
\end{align*}
which is \eqref{eq:endpoint-cumulant}.  Proposition~\ref{prop:endpoint-rate} gives the remaining assertions.

\section{Proof of the functional MDP}\label{sec:MDP}

By \eqref{eq:MDP-deterministic-equivalence} and \eqref{eq:MDP-interpolation}, it is enough to study
\begin{equation}\label{eq:Zn}
 Z_n(t):=
 \frac{S_n(t)-\E S_n(t)}{b_n\sqrt{n\log n}}.
\end{equation}
Set
\begin{equation}\label{eq:Un}
 U_n(t):=
 \frac{S_n(t)-\E S_n(t)}{\sqrt{n\log n}},
 \qquad Z_n=U_n/b_n.
\end{equation}

\medskip
\subsection{Finite-dimensional moderate deviations}

Fix $m\ge1$, $0\le t_1<\cdots<t_m\le1$, and $\lambda\in\R^m$.  Put
\begin{equation}\label{eq:Hnp}
 H_{n,p}:=\sum_{j=1}^m\lambda_j
 \left(
 \1_{\{x_{n,p}G_{n,p}\le t_j\}}
 -\Prob(x_{n,p}G_{n,p}\le t_j)
 \right).
\end{equation}
Then $\E H_{n,p}=0$ and
\[
 \abs{H_{n,p}}\le2\sum_{j=1}^m\abs{\lambda_j}=:M_\lambda.
\]
The scaled logarithmic moment generating function equals
\begin{align}
 &\frac1{b_n^2}
 \log\E\exp\bigg\{
 b_n^2\sum_{j=1}^m\lambda_jZ_n(t_j)
 \bigg\}
 \nonumber\\
 &\qquad=\frac1{b_n^2}\sum_{\sqrt n<p\le n}
 \log\E e^{u_{n,p}H_{n,p}},
 \qquad
 u_{n,p}:=\frac{b_n\log p}{\sqrt{n\log n}}.
 \label{eq:MDP-fdd-start}
\end{align}
Condition \eqref{eq:MDP-scale} implies
\begin{equation}\label{eq:max-unp}
 \max_{p\le n}\abs{u_{n,p}}
 \le b_n\sqrt{\frac{\log n}{n}}
 \longrightarrow0.
\end{equation}

We use the following elementary cumulant estimate.  If $\E H=0$ and $\abs H\le M$, then there are constants $u_0(M)>0$ and $C(M)<\infty$ such that
\begin{equation}\label{eq:cumulant-expansion}
 \abs{\log\E e^{uH}-\tfrac12u^2\E H^2}
 \le C(M)\abs u^3,
 \qquad \abs u\le u_0(M).
\end{equation}
Indeed, Taylor's formula gives
\[
 \E e^{uH}=1+\frac{u^2}{2}\E H^2+R(u),
 \qquad
 \abs{R(u)}\le\frac{e^{\abs uM}M^3}{6}\abs u^3,
\]
and $\log(1+z)=z+O(z^2)$ uniformly for small $z$.

Applying \eqref{eq:cumulant-expansion} in \eqref{eq:MDP-fdd-start}, the quadratic term is
\begin{equation}\label{eq:MDP-quadratic-term}
 \frac1{2n\log n}
 \sum_{\sqrt n<p\le n}(\log p)^2\E H_{n,p}^2.
\end{equation}
Lemma~\ref{lem:mean-covariance} gives
\begin{equation}\label{eq:MDP-quadratic-limit}
 \frac1{2n\log n}
 \sum_{\sqrt n<p\le n}(\log p)^2\E H_{n,p}^2
 \longrightarrow
 \frac12\sum_{i,j=1}^m
 \lambda_i\lambda_jC_\theta(t_i,t_j).
\end{equation}
The total remainder, after division by $b_n^2$, is bounded by
\begin{align}
 \frac{C}{b_n^2}\sum_{p\le n}\abs{u_{n,p}}^3
 &\le
 C\frac{b_n}{(n\log n)^{3/2}}
 \sum_{p\le n}(\log p)^3
 \nonumber\\
 &\le Cb_n\sqrt{\frac{\log n}{n}}
 \longrightarrow0,
 \label{eq:MDP-cubic-remainder}
\end{align}
where we used $\pi(n)\le Cn/\log n$.  Hence
\begin{align}
 &\lim_{n\to\infty}\frac1{b_n^2}
 \log\E\exp\bigg\{
 b_n^2\sum_{j=1}^m\lambda_jZ_n(t_j)
 \bigg\}
 \nonumber\\
 &\qquad=\frac12\sum_{i,j=1}^m
 \lambda_i\lambda_jC_\theta(t_i,t_j).
 \label{eq:MDP-fdd-limit}
\end{align}
The limiting cumulant is finite and differentiable on all of $\R^m$.  Finiteness in a neighborhood of the origin also gives exponential tightness in $\R^m$ by the usual coordinatewise Chernoff bound.  The G\"artner--Ellis theorem \cite[Theorem~2.3.6]{DZ} therefore yields the finite-dimensional MDP, with the Legendre transform of the quadratic form in \eqref{eq:MDP-fdd-limit}.

\medskip
\subsection{Exponential equicontinuity}

The argument here is different from the large deviation compactness argument.  After centering by the exact mean, the process $Z_n$ is no longer monotone, so monotone compactness arguments are unavailable.  We therefore prove exponential equicontinuity directly for centered increments.

Define
\begin{equation}\label{eq:weights-and-times}
 c_{n,p}:=\frac{\log p}{\sqrt{n\log n}},
 \qquad
 T_{n,p}:=x_{n,p}G_{n,p}.
\end{equation}
Then
\begin{equation}\label{eq:Un-empirical}
 U_n(t)=\sum_{\sqrt n<p\le n}c_{n,p}
 \left(
 \1_{\{T_{n,p}\le t\}}
 -\Prob(T_{n,p}\le t)
 \right).
\end{equation}
The maximal weight satisfies
\begin{equation}\label{eq:max-weight}
 \beta_n:=\max_{p\le n}c_{n,p}
 \le\sqrt{\frac{\log n}{n}},
 \qquad
 b_n\beta_n\longrightarrow0.
\end{equation}
For $\delta>0$, define
\begin{equation}\label{eq:Omega-n}
 \Omega_n(\delta)
 :=\sup_{\substack{0\le s\le t\le1\\t-s\le\delta}}
 \abs{U_n(t)-U_n(s)}.
\end{equation}

Fix $\delta\in(0,1]$ and use the regular grid
\[
 0=r_0<r_1<\cdots<r_M=1,
 \qquad
 r_j=(j\delta)\wedge1,
 \qquad
 M\le\delta^{-1}+1.
\]
For the cell $J_j=(r_{j-1},r_j]$, put
\begin{align}
 R_{n,j}
 :=\sup_{u\in J_j}
 \bigg|\sum_{\sqrt n<p\le n}c_{n,p}
 \Big(\1_{\{r_{j-1}<T_{n,p}\le u\}}
 -\Prob(r_{j-1}<T_{n,p}\le u)\Big)\bigg|.
 \label{eq:Rnj}
\end{align}
Every interval of length at most $\delta$ meets at most two adjacent grid cells.  Inside one cell, a centered increment is the difference of two terms appearing in \eqref{eq:Rnj}.  Consequently,
\begin{equation}\label{eq:Omega-grid-bound}
 \Omega_n(\delta)\le4\max_{1\le j\le M}R_{n,j}.
\end{equation}

Let
\begin{equation}\label{eq:Vnj}
 V_{n,j}:=\sum_{\sqrt n<p\le n}c_{n,p}^2
 \Prob(T_{n,p}\in J_j).
\end{equation}
Since each cell has length at most $\delta$, Lemma~\ref{lem:local-variance} gives
\begin{equation}\label{eq:Vnj-bound}
 \limsup_{n\to\infty}\max_{1\le j\le M}V_{n,j}
 \le2c_\theta\delta.
\end{equation}
Proposition~\ref{prop:local-threshold}, applied separately to each cell, yields
\begin{equation}\label{eq:Rnj-expectation}
 \E R_{n,j}\le4\sqrt{V_{n,j}}.
\end{equation}
Thus, for fixed $\delta$,
\begin{equation}\label{eq:Rnj-expectation-uniform}
 \limsup_{n\to\infty}\max_{1\le j\le M}\E R_{n,j}
 \le C\sqrt\delta.
\end{equation}

Fix $\eta>0$.  Since $b_n\to\infty$, \eqref{eq:Rnj-expectation-uniform} implies that, for fixed $\delta$ and all sufficiently large $n$,
\[
 \max_{1\le j\le M}\E R_{n,j}
 \le\frac{b_n\eta}{8}.
\]
By \eqref{eq:Omega-grid-bound},
\[
 \Prob\bigl(\Omega_n(\delta)>b_n\eta\bigr)
 \le\sum_{j=1}^M
 \Prob\bigg(R_{n,j}>\frac{b_n\eta}{4}\bigg).
\]
Apply the concentration estimate \eqref{eq:local-threshold-concentration} with
$x=b_n\eta/8$.  Using the elementary inequality
$\log(1+y)\ge y/(1+y)$, $y\ge0$, we obtain, for a universal constant $C_T$,
\begin{align}
 \Prob\left(R_{n,j}>\frac{b_n\eta}{4}\right)
 &\le C_T\exp\bigg\{
 -\frac{x}{C_T\beta_n}
 \log\bigg(1+
 \frac{x\beta_n}{V_{n,j}+\beta_n\E R_{n,j}}
 \bigg)
 \bigg\}
 \nonumber\\
 &\le C_T\exp\bigg\{
 -\frac{x^2}
 {C_T\{V_{n,j}+\beta_n\E R_{n,j}+x\beta_n\}}
 \bigg\}.
 \label{eq:Rnj-tail}
\end{align}
For fixed $\delta$, relations \eqref{eq:max-weight}, \eqref{eq:Vnj-bound}, and
\eqref{eq:Rnj-expectation-uniform} show that the denominator in the last exponent is at most $C\delta+o(1)$, uniformly in $j$.  Since $M$ is fixed when $\delta$ is fixed, the union bound gives
\begin{equation}\label{eq:MDP-modulus-bound}
 \limsup_{n\to\infty}\frac1{b_n^2}
 \log\Prob\left(
 \Omega_n(\delta)>b_n\eta
 \right)
 \le-\frac{c\eta^2}{\delta}
\end{equation}
for a constant $c>0$ independent of $\delta$ and $\eta$.  Then, letting $\delta\downarrow0$ yields
\begin{equation}\label{eq:MDP-exponential-equicontinuity}
 \lim_{\delta\downarrow0}\limsup_{n\to\infty}
 \frac1{b_n^2}\log\Prob\left(
 \sup_{\abs{t-s}\le\delta}
 \abs{Z_n(t)-Z_n(s)}>\eta
 \right)
 =-\infty.
\end{equation}

At any fixed time, at most one large prime indicator can jump.  Indeed, if
$pG_{n,p}=qG_{n,q}\le n$ for distinct primes $p,q>\sqrt n$, then the common integer is divisible by $pq>n$.  The jump of the centered process $U_n$ is therefore bounded by $2\beta_n$: one $\log p$ term can jump in $S_n$, and the corresponding deterministic mean has a jump of size at most the same weight.  Hence the polygonal interpolation $\widehat Z_n$ satisfies
\begin{equation}\label{eq:Zn-interpolation}
 \norm{\widehat Z_n-Z_n}_\infty
 \le\frac{2\beta_n}{b_n}\longrightarrow0,
\end{equation}
and \eqref{eq:MDP-exponential-equicontinuity} remains valid for $\widehat Z_n$.  The deterministic bound in \eqref{eq:Zn-interpolation} also transfers every finite-dimensional MDP from $Z_n$ to $\widehat Z_n$.

\medskip
\subsection{Completion of the MDP proof}

The finite-dimensional MDPs and \eqref{eq:MDP-exponential-equicontinuity} verify conditions~(a.1)--(a.3) of \cite[Theorem~2.4]{ArconesLD}, now with inverse speed $b_n^{-2}$.  Consequently, $\widehat Z_n$ satisfies an LDP in $C[0,1]$ with speed $b_n^2$ and projective rate
\begin{align}\label{eq:MDP-projective-rate}
 \widetilde J_\theta(f)
 =\sup_{m,\boldsymbol t,\lambda}
 \bigg\{
 \sum_{j=1}^m\lambda_jf(t_j)
 -\frac12\sum_{i,j=1}^m
 \lambda_i\lambda_jC_\theta(t_i,t_j)
 \bigg\}.
\end{align}
We next identify the projective rate.  Let
\[
 \mathcal V_\theta:=\operatorname{span}
 \{C_\theta(t,\cdot):0\le t\le1\},
\]
which is dense in $\cH_\theta$.  For a function $f:[0,1]\to\R$, define on formal finite linear combinations
\[
 L_f\bigg(\sum_{j=1}^m\lambda_jC_\theta(t_j,\cdot)\bigg)
 :=\sum_{j=1}^m\lambda_jf(t_j).
\]
If this prescription is not well defined, there is a zero norm combination on which it is nonzero. Scaling that combination shows immediately that the supremum in \eqref{eq:MDP-projective-rate} is infinite.  Suppose that it is well defined.  Then \eqref{eq:MDP-projective-rate} can be written as
\begin{equation}\label{eq:RKHS-dual-functional}
 \widetilde J_\theta(f)
 =\sup_{g\in\mathcal V_\theta}
 \bigg\{L_f(g)-\frac12\norm{g}_{\cH_\theta}^2\bigg\}.
\end{equation}
If $f\in\cH_\theta$, the reproducing property gives $L_f(g)=\ip{f}{g}_{\cH_\theta}$.  Completing the square and using the density of $\mathcal V_\theta$ yields
\[
 \widetilde J_\theta(f)=\frac12\norm{f}_{\cH_\theta}^2.
\]
If $f\notin\cH_\theta$ and $L_f$ is unbounded on the unit ball of $\mathcal V_\theta$, choose $g_r\in\mathcal V_\theta$ with $\norm{g_r}\le1$ and $L_f(g_r)\to\infty$, and optimizing over scalar multiples gives a lower bound $L_f(g_r)^2/2$, which tends to $+\infty$.  If $L_f$ were bounded, it would extend continuously to $\cH_\theta$. The Riesz representation theorem would then provide $h\in\cH_\theta$ such that
$L_f(g)=\ip{h}{g}_{\cH_\theta}$ for all $g\in\mathcal V_\theta$.  Taking $g=C_\theta(t,\cdot)$ would give $f(t)=h(t)$ for every $t$, contrary to $f\notin\cH_\theta$.  Hence the projective rate is exactly \eqref{eq:MDP-rate}.

Moreover, the rate function is good.  Indeed, for $f\in\cH_\theta$,
\begin{equation}\label{eq:RKHS-equicontinuity}
 \abs{f(t)-f(s)}
 \le\norm{f}_{\cH_\theta}
 \sqrt{C_\theta(t,t)+C_\theta(s,s)-2C_\theta(s,t)}.
\end{equation}
The kernel is continuous by dominated convergence in \eqref{eq:cov-integral}. Thus every RKHS ball is uniformly bounded and equicontinuous.  To verify closedness in $C[0,1]$, let $(f_n)$ lie in a fixed RKHS ball and converge uniformly to $f$.  The RKHS is separable because the span of kernel sections indexed by a countable dense subset of $[0,1]$ is dense.  A subsequence therefore converges weakly in the Hilbert space to some $h$ in the same ball.  Since point evaluations are continuous, $f_n(t)\to h(t)$ along this subsequence for every $t$, while uniform convergence gives $f_n(t)\to f(t)$, so $f=h\in\cH_\theta$.  The Arzel\`a--Ascoli theorem proves compactness of every level set.

The deterministic estimates \eqref{eq:MDP-deterministic-equivalence},
\eqref{eq:MDP-interpolation}, and \eqref{eq:Zn-interpolation} transfer the MDP to $\widehat Y_n$ by exponential equivalence.  Applying the continuous inclusion from $C[0,1]$ with the uniform norm into $D[0,1]$ with the $J_1$-topology, and then using exponential equivalence between $Y_n$ and $\widehat Y_n$, gives the asserted $J_1$-MDP for the step processes.  This proves Theorem~\ref{thm:FMDP}.

Finally, the endpoint map $f\mapsto f(1)$ is continuous.  The contraction principle \cite[Theorem~4.2.1]{DZ} gives Corollary~\ref{cor:endpoint-MDP}.  The squared norm of the evaluation functional at time one is $C_\theta(1,1)=\sigma_\theta^2$.  Thus Cauchy--Schwarz gives
\[
 x^2=f^2(1)
 =\ip{f}{C_\theta(1,\cdot)}_{\cH_\theta}^2
 \le \norm{f}_{\cH_\theta}^2\sigma_\theta^2.
\]
Equality is attained by
$f=xC_\theta(1,\cdot)/\sigma_\theta^2$.  Hence the contracted rate is
$x^2/(2\sigma_\theta^2)$, as stated in \eqref{eq:endpoint-MDP-rate}.

\section{Proof of the functional LIL}\label{sec:LIL}

In this section we prove Theorem~\ref{thm:FLIL}.  The MDP supplies the logarithmic probabilities on a sparse sequence.  Two additional estimates are needed, as the process is not a partial sum process in the parameter n. First, a common-cutoff interpolation bound on geometric blocks handles the upper inclusion. Second, a lacunary independent-block approximation is used for the lower cluster set inclusion.

Recall the normalization $a_n$ from \eqref{eq:LIL-normalization}, and write
\begin{equation}\label{eq:LIL-centered-S}
 C_n(t):=S_n(t)-\E S_n(t),\qquad
 \Xi_n^S(t):=\frac{C_n(t)}{a_n}.
\end{equation}
Let $\widehat C_n$ and $\widehat\Xi_n^S$ denote the corresponding polygonal interpolations.  For the critical choice $b_n=\sqrt{2\log\log n}$, the process $\Xi_n^S$ is exactly the large prime moderate deviation process $Z_n$ from \eqref{eq:Zn}.  Lemma~\ref{lem:deterministic-approximation} and the interpolation bound \eqref{eq:MDP-interpolation} give
\begin{equation}\label{eq:LIL-deterministic-reduction}
 \norm{\widehat\Xi_n-\widehat\Xi_n^S}_\infty
 \le \frac{2C\sqrt n}{a_n}+o(1)\longrightarrow0.
\end{equation}
It is therefore enough to prove the compact LIL for $\widehat\Xi_n^S$.

\medskip
\subsection{Block interpolation}

The next lemma controls the change of the centered large prime process when the endpoint $n$ varies inside a short geometric block.  This is the substitute for the usual increment estimate in partial sum proofs of Strassen's theorem.

\begin{lemma}\label{lem:LIL-geometric-interpolation}
For $\rho\in(0,1/4)$, let $N_m=\lfloor(1+\rho)^m\rfloor$.  Then, on an event of probability one,
\begin{equation}\label{eq:LIL-geometric-interpolation}
 \lim_{\substack{\rho\downarrow0\\ \rho\in\mathbb Q}}
 \limsup_{m\to\infty}
 \max_{N_m\le n\le N_{m+1}}
 \frac{\norm{C_n-C_{N_m}}_\infty}{a_{N_m}}=0.
\end{equation}
\end{lemma}

\begin{proof}
Fix a rational $\rho\in(0,1/4)$ and write $M=N_m$, $N=N_{m+1}$.  For all large $m$,
$1\le N/M\le1+2\rho<2$.  For every prime $p>\sqrt N$, use the Bernoulli variables $\xi_p,\xi_{2p},\ldots,\xi_{\lfloor N/p\rfloor p}$ and independent auxiliary tails to construct a common geometric mark $G_{m,p}$.  Since $pq>N$ for distinct primes $p,q>\sqrt N$, these marks are independent, and for every $M\le n\le N$,
\[
 I_{A_{\lfloor nt\rfloor}}(p)=\1_{\{pG_{m,p}\le nt\}},\qquad 0\le t\le1.
\]
Define the common-cutoff centered process
\[
 \widetilde C_{m,n}(t)
 :=\sum_{\sqrt N<p\le n}\log p
 \left(\1_{\{pG_{m,p}\le nt\}}
 -\Prob(pG_{m,p}\le nt)\right).
\]
The only primes occurring in $C_n-\widetilde C_{m,n}$ lie in $(\sqrt n,\sqrt N]$.  Hence, deterministically,
\begin{equation}\label{eq:LIL-common-cutoff-error}
 \sup_{M\le n\le N}
 \norm{C_n-\widetilde C_{m,n}}_\infty
 \le2\vartheta(\sqrt N)=O(\sqrt N)=o(a_M).
\end{equation}

We decompose
\[
 \widetilde C_{m,n}-\widetilde C_{m,M}
 =U_{m,n}+V_{m,n},
\]
where the shared-prime term is
\begin{align*}
 U_{m,n}(t)
 :=\sum_{\sqrt N<p\le M}\log p
 \left(\1_{\{Mt<pG_{m,p}\le nt\}}-\Prob(Mt<pG_{m,p}\le nt)\right),
\end{align*}
and the new-prime term is
\[
 V_{m,n}(t)
 :=\sum_{M<p\le n}\log p
 \left(\1_{\{pG_{m,p}\le nt\}}
 -\Prob(pG_{m,p}\le nt)\right).
\]

For $U_{m,n}$, set $T_p=pG_{m,p}/N$ and
$c_p=(\log p)/\sqrt{N\log N}$.  Use the regular grid
$r_j=(j\rho)\wedge1$ and, for $J_j=(r_{j-1},r_j]$, define
\[
 R_{m,j}:=\sup_{u\in J_j}
 \bigg|\sum_{\sqrt N<p\le M}c_p
 \left(\1_{\{r_{j-1}<T_p\le u\}}
 -\Prob(r_{j-1}<T_p\le u)\right)\bigg|.
\]
Every interval $(Mt/N,nt/N]$ has length at most
$(N-M)/N\le2\rho$.  Such an interval meets at most three consecutive grid cells, and its centered sum is bounded by a universal constant times $\max_jR_{m,j}$.  The variance proxy of $R_{m,j}$ satisfies, uniformly in $j$,
\begin{equation}\label{eq:LIL-shared-variance}
\begin{aligned}
 \mathbb{V}_{m,j}
 &:=\frac1{N\log N}\sum_{\sqrt N<p\le M}
 (\log p)^2\Prob(T_p\in J_j)\\
 &\le C\rho+o(1),
\end{aligned}
\end{equation}
because the sum is bounded by the corresponding full-row sum in
Lemma~\ref{lem:local-variance}.  Moreover,
$\max_pc_p\le\sqrt{\log N/N}$.  Proposition~\ref{prop:local-threshold} gives
$\max_j\E R_{m,j}\le C\sqrt\rho+o(1)$.  Since
\[
 \frac{a_M}{\sqrt{N\log N}}
 =\left(\frac{2M\log M\log\log M}{N\log N}\right)^{1/2}
 \ge c\sqrt{\log\log M},
\]
the simplified concentration inequality
\eqref{eq:local-threshold-Bernstein}, followed by a union bound over
$O(\rho^{-1})$ grid cells, yields constants $c,C>0$ such that, for every
fixed $u>0$, all sufficiently small $\rho$, and all large $m$,
\begin{equation}\label{eq:LIL-shared-tail}
 \Prob\bigg(
 \max_{M\le n\le N}\norm{U_{m,n}}_\infty>ua_M
 \bigg)
 \le C\rho^{-1}
 \exp\bigg\{-\frac{cu^2}{\rho}\log\log M\bigg\}.
\end{equation}

For $V_{m,n}$, the inequality $N<2M$ implies that a prime $p\in(M,N]$ has no multiple other than $p$ in $[1,N]$.  Thus
\[
 V_{m,n}(t)=\sum_{M<p\le nt}\log p\,(\xi_p-\theta).
\]
After ordering the primes in $(M,N]$, the supremum over $n$ and $t$ is bounded by the maximum of the corresponding partial sums.  The prime number theorem gives, for fixed $\rho$ and all large $m$,
\[
 \sum_{M<p\le N}(\log p)^2\Var(\xi_p)
 \le C\rho N\log N.
\]
The maximal Bernstein inequality in Lemma~\ref{lem:maximal-Bernstein} consequently yields
\begin{equation}\label{eq:LIL-new-prime-tail}
 \Prob\bigg(
 \max_{M\le n\le N}\norm{V_{m,n}}_\infty>ua_M
 \bigg)
 \le2\exp\bigg\{-\frac{cu^2}{\rho}\log\log M\bigg\}
\end{equation}
for all large $m$.

Since $\log\log N_m=\log m+O_\rho(1)$, the right hand sides of
\eqref{eq:LIL-shared-tail} and \eqref{eq:LIL-new-prime-tail} are summable in $m$ whenever the exponent $cu^2/\rho$ is larger than $1$.  We keep a margin and require it to be larger than $2$.  Choose a numerical constant $K$ so large that $cK^2>2$, and apply Borel--Cantelli with $u=K\sqrt\rho$ for rational $\rho$.  On a common probability one event,
\[
 \limsup_{m\to\infty}
 \max_{N_m\le n\le N_{m+1}}
 \frac{\norm{\widetilde C_{m,n}-\widetilde C_{m,N_m}}_\infty}{a_{N_m}}
 \le K\sqrt\rho
 \qquad(\rho\in\mathbb Q\cap(0,1/4)).
\]
Together with \eqref{eq:LIL-common-cutoff-error}, this bound tends to zero as rational $\rho\downarrow0$, proving \eqref{eq:LIL-geometric-interpolation}.
\end{proof}

\medskip
\subsection{Upper cluster set inclusion}

We first prove
\begin{equation}\label{eq:LIL-upper-goal}
 \operatorname{dist}_\infty(\widehat\Xi_n^S,\mathcal K_\theta)
 \longrightarrow0\qquad\text{a.s.}
\end{equation}
Fix $\varepsilon>0$ and set
\[
 F_\varepsilon:=\{f\in C[0,1]:
 \operatorname{dist}_\infty(f,\mathcal K_\theta)\ge\varepsilon\}.
\]
Since $J_\theta$ is good and $\mathcal K_\theta=\{J_\theta\le1/2\}$,
\begin{equation}\label{eq:LIL-rate-gap}
 \inf_{f\in F_\varepsilon}J_\theta(f)>\frac12.
\end{equation}
Indeed, otherwise a sequence in $F_\varepsilon$ with rates tending to at most $1/2$ would have, by compactness of a slightly larger level set, a limit belonging simultaneously to $F_\varepsilon$ and $\mathcal K_\theta$.

Fix rational $\rho\in(0,1/4)$ and let $N_m=\lfloor(1+\rho)^m\rfloor$.  The MDP upper bound with $b_n=\sqrt{2\log\log n}$ and \eqref{eq:LIL-rate-gap} gives some $\eta>0$ such that, for all large $m$,
\begin{equation}\label{eq:LIL-upper-subsequence-summable}
 \Prob(\widehat\Xi_{N_m}^{S}\in F_\varepsilon)
 \le\exp\{-(1+\eta)\log\log N_m\}.
\end{equation}
The right-hand side is summable because $\log\log N_m=\log m+O_\rho(1)$.  Thus
\begin{equation}\label{eq:LIL-upper-on-grid}
 \operatorname{dist}_\infty(
 \widehat\Xi_{N_m}^{S},\mathcal K_\theta)<\varepsilon
 \quad\text{eventually a.s.}
\end{equation}
In particular, $(\widehat\Xi_{N_m}^S)$ is eventually bounded in the uniform norm.

For $N_m\le n\le N_{m+1}$,
\begin{align*}
 \widehat\Xi_n^S-\widehat\Xi_{N_m}^S
 ={}&\frac{\widehat C_n-\widehat C_{N_m}}{a_n}
 +\left(\frac{a_{N_m}}{a_n}-1\right)
 \widehat\Xi_{N_m}^S,
\end{align*}
where hats denote polygonal interpolation.  The interpolation errors are $o(1)$ uniformly on the block, Lemma~\ref{lem:LIL-geometric-interpolation} controls the first term, and
\begin{equation}\label{eq:LIL-normalization-ratio}
 \sup_{N_m\le n\le N_{m+1}}
 \Big|\frac{a_n}{a_{N_m}}-1\Big|
 \le C\rho+o(1).
\end{equation}
Combining these facts with \eqref{eq:LIL-upper-on-grid}, and then letting rational $\rho\downarrow0$ and positive rational $\varepsilon\downarrow0$, proves \eqref{eq:LIL-upper-goal} on a single probability one event.

The convergence of the distance to the compact set $\mathcal K_\theta$ also implies relative compactness.  Indeed, choose $g_n\in\mathcal K_\theta$ with
$\|\widehat\Xi_n^S-g_n\|_\infty\le
\operatorname{dist}_\infty(\widehat\Xi_n^S,\mathcal K_\theta)+1/n$.  Every subsequence of $(g_n)$ has a convergent further subsequence, and the corresponding further subsequence of $(\widehat\Xi_n^S)$ has the same limit.

For the lower bound we use independent lacunary blocks, whose moderate deviations have the same rate.

\begin{lemma}\label{lem:LIL-block-MDP}
Let $\alpha>1$, $n_m=\lfloor e^{m^\alpha}\rfloor$, and
$b_m=\sqrt{2\log\log n_m}$.  For all sufficiently large $m$, define
\begin{equation}\label{eq:LIL-block-process}
 B_m(t):=\frac1{a_{n_m}}
 \sum_{n_{m-1}<p\le n_m}\log p
 \left(
 I_{A_{\lfloor n_mt\rfloor}}(p)
 -\E I_{A_{\lfloor n_mt\rfloor}}(p)
 \right).
\end{equation}
Then the polygonal processes $(\widehat B_m)$ satisfy an LDP in $C[0,1]$ with speed $b_m^2$ and rate $J_\theta$.  Moreover, the processes $(B_m)$ are independent.
\end{lemma}

\begin{proof}
Put $r_m=n_{m-1}/n_m$.  Then $r_m\to0$, and, for all large $m$,
$n_{m-1}^2>n_m$.  Every multiple not exceeding $n_m$ of a prime
$p>n_{m-1}$ belongs to the integer block $(n_{m-1},n_m]$.  Distinct primes in $(n_{m-1},n_m]$ have disjoint multiple sets up to $n_m$, because their product exceeds $n_m$.  The Bernoulli variables used by different values of $m$ lie in disjoint integer blocks.  This proves independence, and it also provides independent geometric marks within every block.

The proof of Theorem~\ref{thm:FMDP} can be repeated with the prime range
$(n_{m-1},n_m]$ in place of $(\sqrt{n_m},n_m]$.  We indicate the points at which the moving lower cutoff must be checked.  For fixed
$0\le t_1<\cdots<t_d\le1$ and $\lambda\in\R^d$, the centered logarithmic moment generating function has the same Taylor expansion as
\eqref{eq:cumulant-expansion}.  Its quadratic term is determined by
\[
 \frac1{n_m\log n_m}
 \sum_{n_{m-1}<p\le n_m}(\log p)^2
 \Cov\!\left(
 \1_{\{(p/n_m)G\le t_i\}},
 \1_{\{(p/n_m)G\le t_j\}}
 \right).
\]
The corresponding full-row sum converges to $C_\theta(t_i,t_j)$ by
Lemma~\ref{lem:mean-covariance}.  The omitted part is bounded uniformly in
$i,j$ by
\[
 \frac{C}{n_m\log n_m}
 \sum_{p\le n_{m-1}}(\log p)^2
 \le C\frac{n_{m-1}\log n_{m-1}}
 {n_m\log n_m}
 \le Cr_m\longrightarrow0.
\]
The cubic remainder is bounded, after division by $b_m^2$, by
\[
 Cb_m\sqrt{\frac{\log n_m}{n_m}}\longrightarrow0,
\]
exactly as in \eqref{eq:MDP-cubic-remainder}.  Hence the finite-dimensional
cumulants converge to the quadratic form in \eqref{eq:MDP-fdd-limit}.

For exponential equicontinuity, deleting primes less than $n_{m-1}$ can only
decrease every local variance proxy.  The maximal normalized weight is at
most $\sqrt{\log n_m/n_m}$ and therefore satisfies
$b_m\sqrt{\log n_m/n_m}\to0$.  The grid argument and
Proposition~\ref{prop:local-threshold} used in
\eqref{eq:MDP-modulus-bound} consequently apply without change.  The
finite-dimensional MDPs and exponential equicontinuity give the asserted
functional MDP by the same process-level theorem as in
Section~\ref{sec:MDP}.
\end{proof}

The next estimate shows that the old-prime contribution omitted from a lacunary block is almost surely negligible.

\begin{lemma}\label{lem:LIL-old-primes}
With $n_m$ as in Lemma~\ref{lem:LIL-block-MDP},
\begin{equation}\label{eq:LIL-old-primes}
 \norm{\widehat\Xi_{n_m}^S-\widehat B_m}_\infty
 \longrightarrow0\qquad\text{a.s.}
\end{equation}
\end{lemma}

\begin{proof}
For all large $m$, $n_{m-1}>\sqrt{n_m}$, and the difference in
\eqref{eq:LIL-old-primes} is the centered contribution of primes
$\sqrt{n_m}<p\le n_{m-1}$, divided by $a_{n_m}$.  These prime indicators are independent within the row $n_m$.  Proposition~\ref{prop:local-threshold}, with the whole interval $J=(0,1]$, applies with weights
$c_{m,p}=\log p/a_{n_m}$.  Its variance proxy satisfies
\begin{align}
 V_m
 &\le\frac1{a_{n_m}^2}
 \sum_{p\le n_{m-1}}(\log p)^2
 \nonumber\\
 &\le C\frac{n_{m-1}\log n_{m-1}}
 {n_m\log n_m\log\log n_m}
 \le C\frac{r_m}{\log\log n_m},
 \label{eq:LIL-old-variance}
\end{align}
and the maximal weight is
\[
 \beta_m\le\frac{\log n_m}{a_{n_m}}
 =O\left(\sqrt{\frac{\log n_m}
 {n_m\log\log n_m}}\right).
\]
Proposition~\ref{prop:local-threshold} therefore bounds the expected supremum of the step process by $O(\sqrt{V_m})=o(1)$.  Passing to polygonal interpolation changes the supremum by at most $2\beta_m=o(1)$.  Fix $u>0$.  For all sufficiently large $m$, the expectation and the interpolation error are both at most $u/4$.  Applying \eqref{eq:local-threshold-Bernstein} with deviation level $u/2$ then gives
\[
 \Prob\left(
 \norm{\widehat\Xi_{n_m}^S-\widehat B_m}_\infty>u
 \right)
 \le C\exp\bigg\{-\frac{cu^2}{V_m+u\beta_m}\bigg\}.
\]
Since
$n_m/n_{m-1}=\exp\{\alpha m^{\alpha-1}+o(m^{\alpha-1})\}$, the last probabilities are summable in $m$.  Borel--Cantelli, first for rational $u>0$ and then by monotonicity, proves \eqref{eq:LIL-old-primes}.
\end{proof}

\medskip
\subsection{Lower cluster set inclusion}

Let $f\in\cH_\theta$ with $\|f\|_{\cH_\theta}<1$.  Choose
\begin{equation}\label{eq:LIL-alpha-choice}
 1<\alpha<\frac1{\|f\|_{\cH_\theta}^2}
\end{equation}
and set $n_m=\lfloor e^{m^\alpha}\rfloor$.  By the MDP lower bound in Lemma~\ref{lem:LIL-block-MDP}, for every $\varepsilon>0$ and every $\delta>0$, all sufficiently large $m$ satisfy
\begin{align}
 \Prob\left(\|\widehat B_m-f\|_\infty<\varepsilon\right)
 &\ge\exp\left\{-b_m^2
 \left(\frac12\|f\|_{\cH_\theta}^2+\delta\right)\right\}
 \notag\\
 &=m^{-\alpha(\|f\|_{\cH_\theta}^2+2\delta)+o(1)}.
 \label{eq:LIL-block-lower-probability}
\end{align}
Choose $\delta$ so small that
$\alpha(\|f\|_{\cH_\theta}^2+2\delta)<1$.  The probabilities in
\eqref{eq:LIL-block-lower-probability} then have a divergent sum.  The events are independent by Lemma~\ref{lem:LIL-block-MDP}, so the second Borel--Cantelli lemma gives
\[
 \|\widehat B_m-f\|_\infty<\varepsilon
 \quad\text{infinitely often, a.s.}
\]
Lemma~\ref{lem:LIL-old-primes} transfers this statement to
$\widehat\Xi_{n_m}^S$.  Taking a countable dense subset of the open unit ball of $\cH_\theta$ and rational $\varepsilon\downarrow0$ shows, on one event of probability one, that every point of that open ball is a cluster point.  Since scalar contractions $rf$, $0<r<1$, converge uniformly to $f$ for every $f\in\mathcal K_\theta$, the closure of the open unit ball in the uniform norm is $\mathcal K_\theta$.  Together with the upper inclusion and relative compactness, this proves that the cluster set is exactly $\mathcal K_\theta$.

\medskip
\subsection{Endpoint and linear functional consequences}

The endpoint evaluation map $e_1(f)=f(1)$ is continuous on $C[0,1]$.  The image of $\mathcal K_\theta$ under $e_1$ is
\[
 [-\|C_\theta(1,\cdot)\|_{\cH_\theta},
 \|C_\theta(1,\cdot)\|_{\cH_\theta}]
 =[-\sigma_\theta,\sigma_\theta],
\]
because $\|C_\theta(1,\cdot)\|_{\cH_\theta}^2=C_\theta(1,1)=\sigma_\theta^2$.  This proves \eqref{eq:endpoint-LIL-limsup} and \eqref{eq:endpoint-LIL-liminf}.  The same argument for an arbitrary continuous linear functional $\ell$ proves \eqref{eq:linear-functional-LIL}.  Finally, \eqref{eq:LIL-deterministic-reduction} transfers the conclusions from the large prime process to the original least common multiple process, and the vanishing interpolation error gives the same cluster set for the step processes in the $J_1$ topology.  This completes the proof of Theorem~\ref{thm:FLIL} and Corollary~\ref{cor:endpoint-LIL}.

\appendix

\section{A local maximal estimate for weighted threshold processes}\label{sec:appendix-local}

The following local maximal estimate for weighted threshold processes supplies the empirical process bound used in the proofs of the functional MDP and the functional LIL.\@  It is a standard consequence of symmetrization, Doob's maximal inequality, and Talagrand's concentration inequality for empirical processes.  The concentration estimate is the product-space inequality given in \cite[Theorem~2.3]{ArconesLD}, a form of Talagrand's inequality \cite{Tala} for suprema of sums of independent, not necessarily identically distributed, random variables.  Since the exact weighted threshold form used here is not isolated in those references, we include the proof for the reader's convenience.

\begin{proposition}\label{prop:local-threshold}
Let $T_1,\ldots,T_N$ be independent real-valued random variables, and let
$0\le c_i\le\beta$.  Fix an interval $J=(a,b]$ and define
\begin{equation}\label{eq:local-threshold-process}
 R_J:=\sup_{a\le u\le b}
 \left|\sum_{i=1}^Nc_i
 \left(
 \1_{\{a<T_i\le u\}}-\Prob(a<T_i\le u)
 \right)\right|
\end{equation}
and
\begin{equation}\label{eq:local-threshold-variance}
 V_J:=\sum_{i=1}^Nc_i^2\Prob(T_i\in J).
\end{equation}
Then
\begin{equation}\label{eq:local-threshold-expectation}
 \E R_J\le4\sqrt{V_J}.
\end{equation}
There is a universal constant $C_T<\infty$ such that, for every $x>0$,
\begin{equation}\label{eq:local-threshold-concentration}
 \Prob(R_J\ge\E R_J+x)
 \le C_T\exp\left\{
 -\frac{x}{C_T\beta}
 \log\left(1+
 \frac{x\beta}{V_J+\beta\E R_J}
 \right)
 \right\}.
\end{equation}
If $\beta=0$, both assertions hold trivially.  In particular, after changing the universal constant if necessary,
\begin{equation}\label{eq:local-threshold-Bernstein}
 \Prob(R_J\ge\E R_J+x)
 \le C_T\exp\bigg\{
 -\frac{x^2}{C_T\{V_J+\beta\E R_J+x\beta\}}
 \bigg\}.
\end{equation}
\end{proposition}

\begin{proof}
For the expectation bound, let $T_1',\ldots,T_N'$ be an independent copy of the sample and let $\varepsilon_1,\ldots,\varepsilon_N$ be independent Rademacher variables, independent of both samples.  The standard symmetrization argument gives
\begin{align}
 \E R_J
 &\le \E\sup_{a\le u\le b}
 \bigg|\sum_{i=1}^Nc_i
 \left(
 \1_{\{a<T_i\le u\}}-
 \1_{\{a<T_i'\le u\}}
 \right)\bigg|
 \nonumber\\
 &\le2\E\sup_{a\le u\le b}
 \bigg|\sum_{i=1}^N\varepsilon_i c_i
 \1_{\{a<T_i\le u\}}\bigg|.
 \label{eq:threshold-symmetrization}
\end{align}
We include the second step explicitly.  Conditional on the two samples, the difference in the first line has the same distribution after multiplication of each pair by an independent sign.  The triangle inequality then separates the original and copied samples, which have the same distribution.

Condition on $(T_i)_{1\le i\le N}$.  List the indices for which $T_i\in J$ in nondecreasing order of $T_i$. Ties may be ordered arbitrarily.  As $u$ increases from $a$ to $b$, the Rademacher sum in \eqref{eq:threshold-symmetrization} takes values among partial sums of the independent mean-zero variables $\varepsilon_i c_i$ in this order.  If several observations are tied, the process only visits a subset of those partial sums, which can only decrease the maximum.  Doob's $L^2$ maximal inequality therefore gives
\begin{align*}
 \E_\varepsilon\sup_{a\le u\le b}
 \bigg|\sum_{i=1}^N\varepsilon_i c_i
 \1_{\{a<T_i\le u\}}\bigg|
 &\le
 \bigg[
 \E_\varepsilon
 \max_{k}\Big|\sum_{r=1}^k\varepsilon_{i_r}c_{i_r}\Big|^2
 \bigg]^{1/2}\\
 &\le2\bigg(
 \sum_{i=1}^Nc_i^2\1_{\{T_i\in J\}}
 \bigg)^{1/2}.
\end{align*}
Taking expectation, using Jensen's inequality, and substituting into
\eqref{eq:threshold-symmetrization} yields
\[
 \E R_J
 \le4\bigg(
 \sum_{i=1}^Nc_i^2\Prob(T_i\in J)
 \bigg)^{1/2}
 =4\sqrt{V_J}.
\]

For the concentration estimate, index the centered functions
\[
 f_{i,u}(t)
 :=c_i\left(
 \1_{\{a<t\le u\}}-\Prob(a<T_i\le u)
 \right),
 \qquad a\le u\le b.
\]
They satisfy $\abs{f_{i,u}}\le c_i\le\beta$, and
\[
 \sup_{a\le u\le b}
 \sum_{i=1}^N\Var(f_{i,u}(T_i))
 \le\sum_{i=1}^Nc_i^2\Prob(T_i\in J)
 =V_J.
\]
The supremum is measurable after restricting $u$ to a countable dense set together with the atoms of the distribution functions.  Equivalently, one may use outer probability and then invoke separability of the right continuous sample paths.  Applying \cite[Theorem~2.3]{ArconesLD} to this centered class gives \eqref{eq:local-threshold-concentration}, after enlarging the universal constant $C_T$ if necessary.  The elementary inequality $\log(1+y)\ge y/(1+y)$, $y\ge0$, then gives \eqref{eq:local-threshold-Bernstein}.
\end{proof}

The LIL interpolation proof also uses the following standard maximal Bernstein inequality.  It follows from Freedman's bounded martingale inequality applied to the partial sums and to their negatives; see \cite{Freedman}.

\begin{lemma}\label{lem:maximal-Bernstein}
Let $X_1,\ldots,X_N$ be independent centered random variables satisfying $|X_i|\le b$ almost surely, and put $V:=\sum_{i=1}^N\Var(X_i)$.  Then, for every $x>0$,
\begin{equation}\label{eq:maximal-Bernstein}
 \Prob\bigg(\max_{1\le k\le N}\bigg|\sum_{i=1}^kX_i\bigg|\ge x\bigg)
 \le2\exp\bigg\{-\frac{x^2}{2(V+bx/3)}\bigg\}.
\end{equation}
\end{lemma}

\section{A dual Euler equation for the functional large deviation rate}\label{sec:appendix-euler}

This appendix gives an Euler-type optimality equation for an extended bounded field dual functional associated with the functional large deviation rate.  The result is supplementary and is not used in the proofs of the  main theorems.  The entropy contraction implies that this extended dual functional never exceeds $I_\theta$.  Equality requires an additional attainment condition.  The purpose of the appendix is to make precise the nonlocal structure of any bounded optimal tilt.

For a bounded measurable function $\eta:[0,1]\to\R$, put
\begin{equation}\label{eq:appendix-Z-eta}
 \mathfrak Z_\eta(x)
 :=q^{m(x)}+\sum_{k=1}^{m(x)}\pi_k e^{\eta(kx)},
 \qquad 0<x\le1,
\end{equation}
and
\begin{equation}\label{eq:appendix-dual-functional}
 \Phi_f(\eta)
 :=\int_0^1\eta(t)\,\dd f(t)
 -\int_0^1\log \mathfrak Z_\eta(x)\,\dd x,
\end{equation}
whenever the first integral is defined.  In particular, if $f$ is absolutely continuous, then it is interpreted as $\int_0^1\eta(t)f'(t)\,\dd t$.

The next proposition is a conditional optimality statement for bounded dual fields.

\begin{proposition}\label{prop:appendix-euler}
Let $f$ be absolutely continuous and nondecreasing.  Suppose that a bounded measurable function $\eta^*$ maximizes the concave functional $\Phi_f$ over bounded measurable functions.  Then
\begin{equation}\label{eq:appendix-euler}
 f'(t)
 =e^{\eta^*(t)}\sum_{k=1}^\infty
 \frac{\pi_k}{k\,\mathfrak Z_{\eta^*}(t/k)}
 \quad\text{for a.e. }t\in(0,1).
\end{equation}
Conversely, if a bounded measurable function $\eta^*$ satisfies \eqref{eq:appendix-euler}, then it is a global maximizer of $\Phi_f$ over bounded measurable functions.
\end{proposition}

\begin{proof}
The map $\eta\mapsto\Phi_f(\eta)$ is concave, because the first term is linear and the second term is the negative of a convex log-sum-exp functional integrated over $x$.  Let $h$ be a bounded measurable perturbation.  For $\varepsilon$ in a bounded interval containing zero, dominated convergence gives
\begin{align}
 \frac{\dd}{\dd\varepsilon}
 \int_0^1\log \mathfrak Z_{\eta+\varepsilon h}(x)\,\dd x
 =&\int_0^1
 \frac{\sum_{k=1}^{m(x)}\pi_k e^{\eta(kx)+\varepsilon h(kx)}h(kx)}
 {\mathfrak Z_{\eta+\varepsilon h}(x)}\,\dd x.
 \label{eq:appendix-Gateaux-before-change}
\end{align}
At $\varepsilon=0$, change variables $t=kx$ in each term.  Since $x\le1/k$ is equivalent to $0\le t\le1$, Tonelli's theorem yields
\begin{align}
 D\Phi_f(\eta)[h]
 =&\int_0^1 h(t)f'(t)\,\dd t
 -\sum_{k=1}^\infty\int_0^{1/k}
 \frac{\pi_k e^{\eta(kx)}h(kx)}{\mathfrak Z_\eta(x)}\,\dd x
 \nonumber\\
 =&\int_0^1 h(t)\bigg[
 f'(t)-e^{\eta(t)}\sum_{k=1}^\infty
 \frac{\pi_k}{k\,\mathfrak Z_\eta(t/k)}
 \bigg]\dd t.
 \label{eq:appendix-first-variation}
\end{align}
The series in \eqref{eq:appendix-first-variation} is finite for almost every $t$: since $\eta$ is bounded and $\mathfrak Z_\eta(x)$ is bounded below by a positive constant depending only on $\norm{\eta}_\infty$ and $\theta$, it is dominated by a constant multiple of $\sum_{k\ge1}\pi_k/k$.

If $\eta^*$ is a maximizer, then $D\Phi_f(\eta^*)[h]=0$ for every bounded measurable $h$.  The fundamental lemma for $L^1$ functions gives \eqref{eq:appendix-euler}.  Conversely, suppose that \eqref{eq:appendix-euler} holds.  For any bounded measurable $\eta$, concavity of the one-dimensional function
$s\mapsto\Phi_f(\eta^*+s(\eta-\eta^*))$ gives
\[
 \Phi_f(\eta)-\Phi_f(\eta^*)
 \le D\Phi_f(\eta^*)[\eta-\eta^*]=0.
\]
Thus $\eta^*$ is a global maximizer over bounded measurable fields.
\end{proof}

\begin{remark}\label{rem:appendix-euler-entropy}
For every bounded measurable $\eta$ and every admissible reduced kernel $a$ with $\mathcal T a=f$, the pointwise Gibbs variational formula \cite[Proposition~1.4.2]{DupuisEllis} gives
\begin{align*}
 &\sum_{k=1}^{m(x)}a_k(x)\eta(kx)
 -\sum_{k=1}^{m(x)}a_k(x)\log\frac{a_k(x)}{\pi_k}
 -a_0(x)\log\frac{a_0(x)}{q^{m(x)}}\le \log\mathfrak Z_\eta(x).
\end{align*}
If $\mathcal T a=f$, then the derivative formula \eqref{eq:path-derivative}, followed by Tonelli's theorem and the change of variables $t=kx$, gives
\[
 \int_0^1\eta(t)\,\dd f(t)
 =\int_0^1\sum_{k=1}^{m(x)}a_k(x)\eta(kx)\,\dd x.
\]
After integration of the pointwise Gibbs inequality, this identity gives
$\Phi_f(\eta)\le \mathcal R_\theta(a)$.  Taking the infimum over $a$ yields $\Phi_f(\eta)\le I_\theta(f)$.  Thus any maximizer of the enlarged dual problem furnishes a supporting hyperplane for the entropy contraction.  When equality is attained, the associated optimal reduced kernel is necessarily the tilted kernel
\begin{equation}\label{eq:appendix-tilted-kernel}
 a_k^\eta(x)=\frac{\pi_k e^{\eta(kx)}}{\mathfrak Z_\eta(x)},\quad 1\le k\le m(x),
 \qquad
 a_0^\eta(x)=\frac{q^{m(x)}}{\mathfrak Z_\eta(x)}.
\end{equation}
The constraint $\mathcal T a^\eta=f$ is then exactly the Euler equation \eqref{eq:appendix-euler} after differentiating the path.
\end{remark}

\begin{remark}\label{rem:appendix-linear-paths}
The Euler equation shows why linear paths are delicate.  If $f_c(t)=ct$ and one tries a constant dual field $\eta(t)\equiv\lambda$, then \eqref{eq:appendix-euler} becomes
\begin{equation}\label{eq:appendix-constant-dual-linear}
 c=e^\lambda\sum_{k=1}^\infty
 \frac{\pi_k}{k\{q^{\lfloor k/t\rfloor}+(1-q^{\lfloor k/t\rfloor})e^\lambda\}},
 \qquad 0<t\le1.
\end{equation}
For $\lambda=0$, the right-hand side is
\[
 \sum_{k\ge1}\frac{\pi_k}{k}=c_\theta,
\]
so the typical linear path $c_\theta t$ is recovered.  For $\lambda\ne0$, the right-hand side in \eqref{eq:appendix-constant-dual-linear} is not constant in $t$.  As $t\downarrow0$, dominated convergence shows that it tends to $c_\theta$.  At $t=1$, its $k$th multiplicative factor relative to $\pi_k/k$ is
\[
 \frac{e^\lambda}{q^k+(1-q^k)e^\lambda},
\]
which is strictly larger than $1$ for $\lambda>0$ and strictly smaller than $1$ for $\lambda<0$.  Hence the value at $t=1$ is respectively larger or smaller than $c_\theta$.  Thus a non-typical endpoint constraint is not generally realized by a straight-line path with a constant dual tilt.  In particular, the endpoint rate is a contraction of the pathwise rate,
\[
 I_\theta^{\mathrm{end}}(c)=\inf\{I_\theta(f):f(1)=c\},
\]
and for the particular linear path one only has
\[
 I_\theta(f_c)\ge I_\theta^{\mathrm{end}}(c),
\]
with equality at least at the typical slope $c=c_\theta$.
\end{remark}


\end{document}